\documentclass[12pt]{amsart}
\usepackage{amsfonts}
\usepackage{amssymb, amscd, amsthm}
\usepackage{latexsym}
\usepackage{multirow}
\usepackage{amsmath}
\usepackage[all]{xy}
\usepackage[dvips]{graphicx}
\usepackage{mathrsfs}
\usepackage{fancyhdr}

\newtheorem{theorem}{Theorem}[section]
\newtheorem{lemma}[theorem]{Lemma}
\newtheorem{thm}[theorem]{Theorem}
\newtheorem{prop}[theorem]{Proposition}
\newtheorem{rem}[theorem]{Remark}
\newtheorem{coro}[theorem]{Corollary}
\newtheorem{defn}[theorem]{Definition}

\newtheorem*{cond}{Condition}
\newtheorem{con/que}[theorem]{Conjecture/Question}

\newcommand{\ra}{\rightarrow}
\newcommand{\mo}{\mathcal{O}}
\newcommand{\mf}{\mathcal{F}}
\newcommand{\mg}{\mathcal{G}}

\newcommand{\me}{\mathcal{E}}
\newcommand{\mi}{\mathcal{I}}
\newcommand{\mk}{\mathcal{K}}
\newcommand{\ms}{\mathcal{S}}
\newcommand{\mt}{\mathcal{T}}

\newcommand{\mm}{\mathcal{M}}

\newcommand{\mw}{\mathcal{W}}
\newcommand{\mv}{\mathcal{V}}

\newcommand{\cl}{\mathcal{L}}
\newcommand{\ts}{\textbf{S}}
\newcommand{\mc}{\mathcal{C}}
\newcommand{\mr}{\mathcal{R}}

\newcommand{\cd}{\mathcal{D}}
\newcommand{\E}{\mathscr{E}}
\newcommand{\F}{\mathscr{F}}
\newcommand{\G}{\mathscr{G}}

\newcommand{\Q}{\mathscr{Q}}
\newcommand{\R}{\mathscr{R}}

\newcommand{\A}{\mathscr{A}}
\newcommand{\B}{\mathscr{B}}
\newcommand{\C}{\mathscr{C}}

\newcommand{\I}{\mathscr{I}}

\newcommand{\Hom}{\operatorname{Hom}}
\newcommand{\Ext}{\operatorname{Ext}}
\newcommand{\Pic}{\operatorname{Pic}}

\newcommand{\Tor}{\operatorname{Tor}}

\def\<{\langle}
\def\>{\rangle}

\newcommand{\z}{\Theta}

\newcommand{\ml}{M(L,0)}
\newcommand{\mlk}{M(L\otimes K_X,0)}
\newcommand{\mli}{M^{int}(L,0)}

\newcommand{\wrn}{W(r,0,n)}

\newcommand{\wtt}{W(2,0,2)}

\newcommand{\lcl}{\lambda_{c^r_n}(L)}

\newcommand{\lrl}{\lambda_r(L)}
\newcommand{\ltl}{\lambda_2(L)}

\newcommand{\crn}{c^r_n}

\newcommand{\cb}{\textbf{CB}}
\newcommand{\ca}{\textbf{CA}}

\newcommand{\ls}{|L|}
\newcommand{\p}{\mathbb{P}}
\newcommand{\bz}{\mathbb{Z}}

\newcommand{\ks}{\mathfrak{S}}

\begin{document}
\fontsize{12pt}{14pt} \textwidth=14cm \textheight=21 cm
\numberwithin{equation}{section}
\title{Strange duality on rational surfaces.}
\author{Yao Yuan}
\address{Yau Mathematical Sciences Center, Tsinghua University, 100084, Beijing, P. R. China}
\email{yyuan@mail.tsinghua.edu.cn; yyuan@math.tsinghua.edu.cn}

\begin{abstract} We study Le Potier's strange duality conjecture on a rational surface.  We focus on the case involving the moduli space of rank 2 sheaves with trivial first Chern class and second Chern class 2, and the moduli space of 1-dimensional sheaves with determinant $L$ and Euler characteristic 0.   We show the conjecture for this case is true under some suitable conditions on $L$, which applies to $L$ ample on any Hirzebruch surface $\Sigma_e:=\p(\mo_{\p^1}\oplus\mo_{\p^1}(e))$ except for $e=1$.  When $e=1$, our result applies to $L=aG+bF$ with $b\geq a+[a/2]$, where $F$ is the fiber class, $G$ is the section class with $G^2=-1$ and $[a/2]$ is the integral part of $a/2$. 

\end{abstract}
\maketitle
\section{Introduction.}
In this whole paper, $X$ is a rational surface over the complex number $\mathbb{C}$, with $K_X$ the canonical divisor and $H$ the polarization such that the intersection number $K_X.H<0$.  We use the same letter to denote both the line bundles and the corresponding divisor classes, but we write $L_1\otimes L_2$, $L^{-1}$ for line bundles while $L_1+L_2$, $-L$ for the corresponding divisor classes.  Denote by $L_1.L_2$ the intersection number of $L_1$ and $L_2$.  $L^2:=L.L$.

Let $K(X)$ be the Grothendieck group of coherent sheaves over $X$.  Define a quadratic form $(u,c)\mapsto \<u,c \>:=\chi(u\otimes c)$ on $K(X)$, where $\chi(-)$ is the holomorphic Euler characteristic and $\chi(u\otimes c)=\chi(\mf\otimes^{L}\mg)$ for any $\mf$ of class $u$, $\mg$ of class $c$ and $\otimes^L$ the flat tensor.  

For two elements $c,u\in K(X)$ orthogonal to each other with respect to $\<,\>$, we have $M_X^H(c)$ and $M_X^H(u)$ the moduli spaces of $H$-semistable sheaves of classes $c$ and $u$ respectively.  If there are no strictly semistable sheaves of classes $c$ ($u$, resp.), then over $M_X^H(c)$ ($M_X^H(u)$, resp.) there is a well-defined line bundle $\lambda_c(u)$ ($\lambda_u(c)$, resp.) called determinant line bundle associated to $u$ ($c$, resp.).  If there are strictly semistable sheaves of class $u$, one needs more conditions on $c$ to get $\lambda_u(c)$ well-defined (see Ch 8 in \cite{HL}).

Let $c,u\in K(X)$.  Assume both moduli spaces $M^X_H(c)$ and $M^X_H(u)$ are non-empty and the determinant line bundles $\lambda_c(u)$ and $\lambda_u(c)$ are well-defined  over $M^X_H(c)$ and $M^X_H(u)$, respectively.  According to \cite{LPst} (see \cite{LPst} p.9), if the following $(\bigstar)$ is satisfied,

($\bigstar $) \emph{for all $H$-semistable sheaves $\mf$ of class $c$ and $H$-semistable sheaves $\mg$ of class $u$ on $X$, $\Tor^i(\mf,\mg)=0$ $\forall~i\ge 1$, and $H^2(X,\mf\otimes \mg)=0$.}

then there is a canonical map
\begin{equation}\label{inmap} SD_{c,u}:H^0(M^X_H(c),\lambda_c(u))^{\vee}\ra H^0(M^X_H(u),\lambda_u(c)).\end{equation}
The strange duality conjecture asserts that $SD_{c,u}$ is an isomorphism.

Strange duality conjecture on curves was at first formulated (in \cite{Bea} and \cite{DT}) and has been proved (see \cite{MO1}, \cite{Bel2}). 
Strange duality on surfaces does not have a general formulation so far.  There is a special formulation due to Le Potier (see \cite{LPst} or \cite{Da2}). 
In this paper we choose $u=u_L:=[\mo_X]-[L^{-1}]+\frac{(L.(L+K_X))}2[\mo_x]$ with $x$ a single point in $X$, and $c=c^2_2:=2[\mo_{X}]-2[\mo_x]$.  Then ($\bigstar$) is satisfied and $SD_{c,u} $ is well-defined.  We prove the following theorem.\begin{thm}[Corollary \ref{ok}]\label{intro}Let $X$ be a Hirzebruch surface $\Sigma_e$ and $L=aG+bF$ with $F$ the fiber class and $G$ the section such that $G^2=-e$.  Then the strange duality map $SD_{c_2^2,u_L}$ as in (\ref{inmap}) is an isomorphism for the following cases.
\begin{enumerate}
\item $\min\{a,b\}\leq1$;
\item $\min\{a,b\}\geq2$, $e\neq 1$, $L$ ample;
\item $\min\{a,b\}\geq2$, $e=1$, $b\geq a+[a/2]$ with $[a/2]$ the integral part of $a/2$.
\end{enumerate}
\end{thm}
Although strange duality on surfaces is a very interesting problem, there are very few cases known.  Our result adds to previous work by the author (\cite{Yuan1},\cite{Yuan5}) and others (\cite{Abe}, \cite{Da1}, \cite{Da2}, \cite{GY}, \cite{MO2}, \cite{MO3}, \cite{MO4}). 

Especially, in \cite{Yuan5} we proved $SD_{c_2^2,u_L}$ is an isomorphism when $X=\p^2$.  The limitation of the method in \cite{Yuan5} is that: we have used Fourier transform on $\p^2$ which does not behave well on other rational surfaces.  In this paper we use a new strategy.  Actually we show the strange duality map $SD_{c^2_2,u_L}$ is an isomorphism under a list of conditions, and then check that all these conditions are fulfilled for cases in Theorem \ref{intro}.  So Theorem \ref{intro} is an application of our main theorem (Theorem \ref{maino}) to Hirzebruch surfaces and there are certainly more applications to other rational surfaces.

The structure of the paper is arranged as follows.  In \S~\ref{pre} we give preliminaries, including some useful properties of $M_X^H(c_2^2)$ (in \S~\ref{bas} and \S~\ref{ssos}) and a brief introduction to determinant line bundles and the set-up of strange duality (in \S~\ref{deter}).  \S~\ref{mare} is the main part.  In \S~\ref{alp} and \S~\ref{ssbeta} we prove the strange duality map is an isomorphism under a list of conditions; in \S~\ref{app} we show the main theorem (Theorem \ref{maino}) applies to cases on Hirzebruch surfaces.  Although the argument in \S~\ref{app} takes quite much space, the technique used there is essentially a combination of those in \cite{Yuan4} and \cite{Yuan5}.      

\begin{flushleft}\emph{Notations.} Let $\mf$, $\mg$ be two sheaves.
\begin{itemize}
\item $c_i(\mf)$ is the i-th Chern class of $\mf$;
\item $\chi(\mf)$ is the Euler characteristic of $\mf$;
\item $h^i(\mf)=dim~H^i(\mf)$;
\item $\text{ext}^i(\mf,\mg)=dim~\Ext^i(\mf,\mg)$, $\text{hom}(\mf,\mg)=dim~\Hom(\mf,\mg)$ and $\chi(\mf,\mg)=\sum_{i\geq0}(-1)^i\text{ext}^i(\mf,\mg)$;   
\item $Supp(\mf)$ or $C_{\mf}$ is the support of 1-dimensional sheaf $\mf$
\end{itemize}
\end{flushleft}
\textbf{Acknowledgements.}
The author was supported by NSFC grant 11301292.  

\section{Preliminaries.}\label{pre}
For any line bundle $L$ on $X$, define $u_L:=[\mo_X]-[L^{-1}]+\frac{(L.(L+K_X))}2[\mo_x]\in K(X)$ with $x$ a single point in $X$.
It is easy to check $u_{\mo_X}=0$ and $u_{L_1}+u_{L_2}=u_{L_1\otimes L_2}$.  If $L$ is nontrivially effective, i.e. $L\not\cong\mo_X$ and $H^0(L)\neq0$, let $\ls$ be the linear system, 
then $u_L$ is the class of 1-dimensional sheaves supported at curves in $\ls$ and of Euler characteristic 0.  

For $L$ nontrivially effective, denote by $\ml$ the moduli space $M^H_X(u_L)$. 
In fact a sheaf $\mf$ of class $u_L$ is semistable  (stable, resp.) if and only if $\forall~\mf'\subsetneq \mf$, $\chi(\mf')\leq0$ ($\chi(\mf')<0$, resp.).  Hence $\ml$ does not depend on the polarization $H$.  We ask $M(\mo_X,0)$ to be a single point standing for the zero sheaf.

Let $\crn=r[\mo_{X}]-n[\mo_x]\in K(X)$ with $x$ a single point on $X$.
Denote by $\wrn$ the moduli space $M_X^H(\crn)$ (but $\wrn$ might depend on $H$).  
In this paper we mainly focus on $\wtt$ for $X$ a rational surface.
 
For any $L,r,n$, $u_L$ and $\crn$ are orthogonal with respect to the quadratic form $\<,\>$ on $K(X)$.
\subsection{Some basic properties of $W(2,0,2)$.}\label{bas}
\begin{defn}\label{hgen}We say the polarization $H$ is \textbf{$c_2^2$-general}, if for any $\xi\in H^2(X,\bz)\cong\Pic(X)$ such that $\xi.H=0$ and $\xi^2\geq -2$, we have $\xi=0$.
\end{defn}
\begin{rem}\label{bio}Since $K_X.H<0$, $\xi.H=0\Rightarrow \xi^2\leq -2$ for any $\xi\in \Pic(X)$.  This is because $H^0(\mo_X(\pm\xi))=0$ by $\xi.H=0$ and $H^2(\mo_X(\pm\xi))=H^0(\mo_X(K_X\mp\xi))^{\vee}=0$ by $(K_X\mp\xi).H<0$, hence $\chi(\mo_X(\xi)\oplus\mo_X(-\xi))=2+\xi^2\leq0$.
\end{rem}
 \begin{lemma}\label{ssl}Let $\mf$ be an $H$-semistable sheaf in class $c_2^2$.  If $\mf$ is not locally free, then it is strictly semistable and S-equivalent to $\mi_x\oplus\mi_y$ with $x,y$ two single points on $X$.  Moreover, if $H$ is $c_2^2$-general, then $\mf$ is $H$-stable if and only if $\mf$ is locally free.
 \end{lemma}
 \begin{proof}First assume $\mf$ is not locally free, then its reflexive hull $\mf^{\vee\vee}$ is locally free of class $c^2_i$ with $i=1$ or 0.  $H^2(\mf^{\vee\vee})\cong H^2(\mf)\cong\Hom(\mf,K_X)^{\vee}=0$ by $K_X.H<0$ and the semistability of $\mf$.  Hence $dim~H^0(\mf^{\vee\vee})\geq \chi(\mf^{\vee\vee})=2-i>0.$  Therefore either $\mf^{\vee\vee}\cong \mo_X^{\oplus2}$ or $\mf^{\vee\vee}$ lies in the following sequence
 \begin{equation}\label{refhu}0\ra\mo_X\xrightarrow{\jmath}\mf^{\vee\vee}\ra \mi_x\ra 0,\end{equation}
 where $\mi_x$ is the ideal sheaf of some single point $x$ on $X$.  
 
 If $\mf^{\vee\vee}$ lies in (\ref{refhu}), then we have
 \begin{equation}\label{fref}0\ra\mf\ra\mf^{\vee\vee}\xrightarrow{p}\mt_1\ra0,\end{equation}
 where $\mt_1$ is a 0-dimensional sheaf with $\chi(\mt_1)=1$ and hence $\mt_1\cong\mo_y$ for some single point $y\in X$.  Compose maps $\jmath$ in (\ref{refhu}) and $p$ in (\ref{fref}), the map $p\circ \jmath:\mo_X\ra \mt_1$ is not zero because otherwise $\mo_X$ would be a subsheaf of $\mf$.  Therefore $p\circ\jmath$ is surjective with kernel isomorphic to $\mi_y$ which is a subsheaf of $\mf$ destabilizing $\mf$.  Hence $\mf$ is not stable and S-equivalent to $\mi_x\oplus\mi_y$.
 
If $\mf^{\vee\vee}\cong\mo_X^{\oplus2}$, then we have the following exact sequence
 \[0\ra\mf\ra\mo_X^{\oplus2}\ra\mt_2\ra0,\]
where $\mt_2$ is a 0-dimensional sheaf with $\chi(\mt_2)=2$.  We also have
\[0\ra\mo_x\ra\mt_2\ra\mo_y\ra 0,\]
where $x,y$ are two single points on $X$ (it is possible to have $x=y$).  Hence we have the following diagram
\[\xymatrix@R=0.4cm@C=0.5cm{&0\ar[d] &0\ar[d] &0\ar[d]&
\\ 0\ar[r]&\mi_x\ar[r]\ar[d]&\mo_X\ar[r]\ar[d]&\mo_x\ar[r]\ar[d]&0\\
0\ar[r]&\mf\ar[r]\ar[d]&\mo_X^{\oplus2}\ar[r]\ar[d]&\mt_2\ar[r]\ar[d]&0\\
0\ar[r]&\mi_y\ar[r]\ar[d]&\mo_X\ar[r]\ar[d]&\mo_y\ar[r]\ar[d]&0\\
&0 &0&0 &
}.\]
Hence $\mf$ is $S$-equivalent to $\mi_x\oplus\mi_y$.

Now assume $H$ is $c^2_2$-general.  We only need to show that any semistable bundle $\mf$ of class $c_2^2$ is stable.  If $\mf$ is strictly semistable, then we have the following sequence
 \[0\ra \mi_Z(\xi)\ra \mf\ra\mi_W(-\xi)\ra0,\]
 where $\xi.H=0$ and $\mi_Z,\mi_W$ are ideal sheaves of  0-dimensional subschemes $Z,W$ of $X$ such that the length $len(Z)=len(W)=1+\xi^2/2\geq 0$.  Since $H$ is $c^2_2$-general, $\xi=0$ and $\mi_Z$ is a subsheaf of $\mf$.
 Hence so is $\mo_X$ because $\mf$ is locally free,
 which is a contradiction since $H^0(\mf)=0$ by semistability.  Hence $\mf$ is stable.  The lemma is proved.
\end{proof}

Denote by $\ts$ the closed subset of $\wtt$ consisting of non locally free sheaves, then set-theoretically $\ts$ is isomorphic to the second symmetric power $X^{(2)}$ of $X$ by Lemma \ref{ssl}.  $\ts$ is of codimension 1 in $\wtt$.  In \S\ref{ssos} we will give a scheme-theoretic structure of $\ts$ and show that it is a divisor associated to some line bundle.

\begin{rem}\label{ngx}If $H$ is not $c^2_2$-general, then all strictly semistable vector bundle are S-equivalent to $\mo_X(\xi)\oplus\mo_X(-\xi)$ with $\xi\in \Pic(X)$, $\xi.H=\xi.K_X=0$ and $\xi^2=-2$.
\end{rem}

\subsection{Determinant line bundles and strange duality.}\label{deter}~~

To set up the strange duality conjecture, we briefly introduce so-called determinant line bundles on the moduli spaces of semistable sheaves.   We refer to Chapter 8 in \cite{HL} for more details.  

For a Noetherian scheme $Y$, we denote by $K(Y)$ the Grothendieck groups of coherent sheaves on $Y$ and $K^0(Y)$ be the subgroup of $K(Y)$ generated by locally free sheaves.  Then $K^0(X)=K(X)$ since $X$ is smooth and projective.  

Let $\E$ be a flat family of coherent sheaves of class $c$ on $X$ parametrized by a noetherian scheme $S$, then $\E\in K^0(X\times S)$.
Let $p:X\times S\to S$, $q:X\times S\to X$ be the projections.
Define $\lambda_\E:K(X)\to \Pic(S)$ to be the composition of the following homomorphisms:
\begin{equation}\label{dlb}
\xymatrix@C=0.3cm{
  K(X)=K^0(X) \ar[rr]^{~~q^{*}} && K^0(X\times S) \ar[rr]^{.[\E]} && K^0(X\times
S) \ar[rr]^{~~~R^{\bullet}p_{*}} && K^0(S)\ar[rr]^{det^{-1}} &&
\Pic(S),}\end{equation}
where $q^*$ is the pull-back morphism, $[\F].[\G]:=\sum_i (-1)^i[\Tor^i(\F,\G)]$, 
and $R^{\bullet}p_{*}([\F])=\sum_i(-1)^i[R^ip_{*}\F].$  Proposition 2.1.10 in \cite{HL} assures that $R^{\bullet}p_{*}([\F])\in K^0(S)$ for any $\F$ coherent and $S$-flat .  

For any $u\in K(X)$, $\lambda_{\E}(u)\in \Pic(S)$ is called the \textbf{determinant line bundle} associated to $u$ induced by the family $\E$.  Notice that the definition we use here is dual to theirs in \cite{HL}. 

Let $S=M_X^H(c)$, 
then there is in general no such universal family $\E$ over $X\times M^H_X(c)$, and even if it exists, there is ambiguity caused by tensoring with the pull-back of a line bundle on $M^H_X(c)$.  Thus to get a well-defined determinant line bundle $\lambda_c(u)$ over $M^H_X(c)$, we need look at the good $GL(V)$-quotient $\Omega(c)\ra M_X^H(c)$ with $\Omega(c)$ an open subset of some Quot-scheme and there is a universal quotient $\widetilde{\E}$ over $X\times \Omega(c)$.  $\lambda_c(u)$ is then defined by descending the line bundle $\lambda_{\widetilde{\E}}(u)$ over $\Omega(c)$.  $\lambda_{\widetilde{\E}}(u)$ descends if and only if it satisfies the ``descent condition" (see Theorem 4.2.15 in \cite{HL}), which implies that $u$ is orthogonal to $c$ with respect to the quadratic form $\<~,~\>$.  Hence the homomorphism $\lambda_c$ is only defined over a subgroup of $K(X)$.

Now we focus on $\ml$ and $\wrn$.  As we have seen, $u_L$ is orthogonal to $\crn$ for any $L,r,n$.  

Let $\lcl$ be the determinant line bundle associated to $u_L$ over (an open subset of) $\wrn$.  We denote simply by $\lrl$ if $r=n$.  By checking the descent condition we see that $\ltl$ is always well-defined over the stable locus $\wtt^s$ and $\ts$, hence it is well-defined over all $\wtt$ if $H$ is $c_2^2$-general.  If $H$ is not $c_2^2$-general, then $\ltl$ is well-defined over point $[\mo_X(-\xi)\oplus\mo_X(\xi)]$ if and only if $\xi.L=0$.  We denote by $\wrn^L$ the biggest open subset of $\wrn$ where $\lcl$ is well-defined.  Notice that the stable locus $\wrn^s\subset\wrn^L$.  By Remark \ref{ngx}, $\wtt^L=\wtt^{L\otimes K_X}$.
 
On the other hand, let $\lambda_{L}(\crn)$ be the determinant line bundle associated to $\crn$ over $\ml$, then $\lambda_{L}(\crn)$ is always well-defined over the whole moduli space.  We have the following proposition which is analogous to Theorem 2.1 in \cite{Da2}.
\begin{prop}\label{glob}(1) There is a canonical section, unique up to scalars, $\sigma_{\crn,u_L}\in H^0(\wrn^L\times\ml,\lcl\boxtimes\lambda_L(\crn))$ whose zero set is 
$$\mathtt{D}_{\crn,u_L}:=\big\{([\mf],[\mg])\in   \wrn^L\times \ml \bigm| h^0(X,\mf\otimes \mg)=h^1(X,\mf\otimes\mg)\ne 0\big\}.$$

(2) The section $\sigma_{\crn,u_L}$ defines a linear map up to scalars
\begin{equation}
\label{SDmap}
SD_{\crn,u_L}:H^0(\wrn^L,\lcl)^\vee \to H^0(\ml,\lambda_L(\crn)).
\end{equation}

(3) Denote by $\sigma_{\mf}$ the restriction of $\sigma_{\crn,u_L}$ to $\{\mf\}\times\ml$.  $\sigma_{\mf}$ only depends on the S-equivalence class of $\mf$.  

(4) If $\sigma_{\crn,u_L}$ is not identically zero, then by assigning $\mf$ to $\sigma_{\mf}$ we get a rational map $\Phi: \wrn^L\ra\p(H^0(\ml,\lambda_L(\crn)))$.  Similarly we have a rational map $\Psi:\ml\ra\p(H^0(\wrn^L,\lcl))$.  Moreover If the image of $\Phi$ is not contained in a hyperplan, then $SD_{\crn,u_L}$ is injective; if the image of $\Psi$ is not contained in a hyperplan, then $SD_{\crn,u_L}$ is surjective. 

\end{prop}
\begin{proof}The proof of Theorem 2.1 in \cite{Da2} also applies to our case although the surface may not be $\p^2$.  For statement (3) and (4), one can also see Lemma 6.13 and Proposition 6.17 in \cite{GY}.
\end{proof}

The map $SD_{\crn,u_L}$ in (\ref{SDmap}) is call the \textbf{strange duality map}, and Le Potier's strange duality is as follows (also see Conjecture 2.2 in \cite{Da2})
\begin{con/que}If both $\wrn^L$ and $\ml$ are non-empty, then is $SD_{\crn,u_L}$ an isomorphism? 
\end{con/que}
 
We denote by $\z_L$ the determinant line bundle associated to $c^1_0=[\mo_{X}]$ on $\ml$.  
Then $\z_L$ defines a divisor $D_{\z_L}$ which consists of sheaves with non trivial global sections.  Since $\lambda_L$ is a group homomorphism, by Proposition 2.8 in \cite{LP2}, we have that $\lambda_L(\crn)\cong\z_L^{\otimes r}\otimes\pi^{*}\mo_{\ls}(n)=:\z^r_L(n)$ where $\pi:\ml\ra\ls$ sends each sheaf to its support.

In this paper we study the following strange duality map for $X$ a rational surface
\begin{equation}\label{sdmt}
SD_{2,L}:H^0(\wtt^L,\ltl)^{\vee}\ra H^0(\ml,\z_L^2(2)).
\end{equation}
   


\subsection{Scheme-theoretic structure of $\ts$ on $\wtt$.}\label{ssos}
~~

$\ts$ consists of non locally free sheaves in $\wtt$.  Recall we have a good quotient $\rho:\Omega_2\ra\wtt$.  Let $\widetilde{\ts}=\rho^{-1}(\ts).$

Set-theoretically $\ts\cong X^{(2)}$.  Let $\Delta\subset X^{(2)}$ be the singular locus and $\Delta\cong X$.  Define $\ts^o=\ts-\Delta$,
$\wtt^o=\wtt^L-\Delta$, $\widetilde{\ts^o}=\rho^{-1}(\ts^o)$ and $\Omega^o_2=\rho^{-1}(\wtt^o)$.  Let $\F$ ($\F^o$, resp.) be the universal quotient over $X\times\Omega_2$ ($X\times\Omega_2^o$, resp.).
We then have the following proposition due to Abe (see Section 3.4 and Section 5.2 Proposition 3.7 and Proposition 5.2 in \cite{Abe})

\begin{prop}\label{stts}(1) The second Fitting ideal $\mathtt{Fitt}_2(\F^o)$ of $\F^o$ defines a smooth closed subscheme $\widehat{\ts^o}$ of $X\times\Omega_2^o$ supported at the set 
$$\{(x,[q:\mo_X(-mH)\otimes  V\twoheadrightarrow\mf])|dim_{k(x)}\mf_x\otimes k(x)>2\}\subset X\times\Omega_2^o.$$
 i.e. $\widehat{\ts^o}$ consists of points $(x,[q:\mo_X(-mH)\twoheadrightarrow\mf])$ such that $\mf_x$ is not free.  
 
(2) We have a surjective map $p_{\Omega}:\widehat{\ts^o}\ra\widetilde{\ts^o}$ induced by the projection $p_{\Omega}:X\times\Omega_2\ra\Omega_2$.  We give a scheme structure of $\widetilde{\ts^o}$ by letting its defining ideal be the kernel of $\mo_{\Omega_2^o}\ra p_{\Omega*}\mo_{\widehat{\ts^o}}$.  Then $\widetilde{\ts^o}$ is a normal crossing divisor in $\Omega_2^o$ with $\widehat{\ts^o}\ra\widetilde{\ts^o}$ the normalization.

(3) The line bundle associated to the divisor $\widetilde{\ts^o}$ on $\Omega_2^o$ is $\lambda_{\F^o}(u_{K_X^{-1}})$.
\end{prop}
\begin{proof}Sheaves in $\widetilde{\ts^o}$ are all quasi-bundles (see Definition 2.1 in \cite{Abe}), hence Abe's argument in Section 3.4 in \cite{Abe} gives Statement (1) and (2).  Notice that our notations are slightly different from his.  

For Statement (3), by Proposition 5.2 in \cite{Abe} we know that $\mo_{\Omega_2^o}(\widetilde{\ts^o})\cong \lambda_{\F^o}([K_X])^{-1}\otimes \lambda_{\F^o}([\mo_X])^{-1}$. 
We can see that $\lambda_{\F^o}(u_{K_X^{-1}})\cong \lambda_{\F^o}([K_X])^{-1}\otimes\lambda_{\F^o}([\mo_X])$.  But $\lambda_{\F^o}([\mo_X])\cong\mo_{\Omega_2^o}$ since 
$H^i(\mf)=0$ for $i=0,1,2$ and $\mf$ semistable of class $c_2^2$.  Hence the proposition.
\end{proof}
\begin{coro}\label{store}Let $\ts$ have the scheme-theoretic structure as the closure of $\ts^o$ in $\wtt$.  Then $\ts$ is a divisor associated to the line bundle $\lambda_2(K_X^{-1})$ on $\wtt$.  Moreover $\ts$ is an integral scheme.
\end{coro}
\begin{proof}By Proposition \ref{stts}, $\ts^o$ is a divisor associated to $\lambda_2(K_X^{-1})$ restricted on $\wtt^o$.  $\ts$ is the closure of $\ts^o$ in $\wtt^L$.  
Since $K_X.H<0$, $\wtt$ is normal, Cohen-Macaulay and of pure dimension 5, hence the section given by $\ts^o$ extends to a section of $\lambda_2(K_X^{-1})$ on $\wtt^L$ with divisor $\ts$. 

We have a morphism $\varphi:X^{(2)}\ra\ts$ sending $(x,y)$ to $\mi_x\oplus\mi_y$, which is bijective.  Hence $\ts$ is irreducible.  $\widetilde{\ts^o}$ is reduced, hence so are $\ts^o$ and $\ts$.  Thus $\ts$ is an integral scheme. 
\end{proof}

\begin{lemma}\label{slow}For any line bundle $L$, the map $H^0(\ts,\lambda_2(L)|_{\ts})\xrightarrow{\varphi^{*}} H^0(X^{(2)},\varphi^{*}\lambda_2(L))$ induced by $\lambda_2(L)|_{\ts}\ra \varphi_{*}\varphi^{*}\lambda_2(L)$ is injective.  Moreover $H^0(X^{(2)},\varphi^{*}\lambda_2(L))\cong (H^0(X,L)^{\otimes 2})^{\mathfrak{S}_2}\cong S^2H^0(X,L)$ where $\mathfrak{S}_n$ is the n-th symmetric group. 
\end{lemma}
\begin{proof}Let $\Delta\subset X^2$ be the diagonal, and $\mi_{\Delta}$ is the ideal sheaf of $\Delta$ in $X^2$.  Let $pr_{i,j}$ be the projection from $X^n$ to the product $X^2$ of the i-th and j-th factors.  Then $pr_{1,2}^{*}\mi_{\Delta}\oplus pr_{1,3}^{*}\mi_{\Delta}$ gives a family of ideal sheaves on $X^3$ and induces a morphism $\widetilde{\varphi}:X^2\ra \wtt$ with image $\ts$.  $\widetilde{\varphi}$ is $\ks_2$-invariant, hence factors through $X^2\ra X^{(2)}$ and gives the map $\varphi:X^{(2)}\ra\ts$.  The morphism $\varphi$ is bijective and $\ts$ is reduced, hence the map $\varphi^{\natural}:\mo_{\ts}\ra \varphi_{*}\mo_{X^{(2)}}$ is injective.  Hence so is the map $\lambda_2(L)|_{\ts}\ra \varphi_{*}\varphi^{*}\lambda_2(L)$ and therefore $H^0(\ts,\lambda_2(L)|_{\ts})\xrightarrow{\varphi^{*}} H^0(X^{(2)},\varphi^{*}\lambda_2(L))$ is injective.  

Obviously $H^0(X^{(2)},\varphi^{*}\lambda_2(L))\cong (H^0(X^2,\widetilde{\varphi}^{*}\lambda_2(L)))^{\mathfrak{S}_2}$.  It will suffice to show that $H^0(X^2,\widetilde{\varphi}^{*}\lambda_2(L))\cong H^0(X,L)^{\otimes 2}$.  By the basic properties (see Lemma 8.1.2 and Theorem 8.1.5 in \cite{HL}) of the determinant line bundle, we have $\widetilde{\varphi}^{*}\lambda_2(L)\cong\lambda_{pr_{1,2}^{*}\mi_{\Delta}\oplus pr_{1,3}^{*}\mi_{\Delta}}(u_L)\cong \lambda_{pr_{1,2}^{*}\mi_{\Delta}}(u_L)\otimes \lambda_{pr_{1,3}^{*}\mi_{\Delta}}(u_L)\cong \lambda_{\mi_{\Delta}}(L)^{\boxtimes 2}$.  Obviously $\lambda_{\mi_{\Delta}}(L)\cong L$, thus $H^0(X^2,\widetilde{\varphi}^{*}\lambda_2(L))\cong H^0(X,\lambda_{\mi_{\Delta}}(L))^{\otimes 2}\cong H^0(X,L)^{\otimes 2}$.  Hence the lemma.
\end{proof}

The line bundle $L^{\boxtimes n}$ on $X^n$ is $\ks_n$-linearized and descends to a line bundle on $X^{(n)}$, which we denote by $L_{(n)}$.
So $\varphi^{*}\lambda_2(L)\cong L_{(2)}$ on $X^{(2)}$.  Denote also by $L_{(n)}$ the pullback of $L_{(n)}$ to $X^{[n]}$ via the Hilbert-Chow morphism, where $X^{[n]}$ is the Hilbert scheme of $n$-points on $X$. 

\section{Main result on $SD_{2,L}$.}\label{mare} 
Let $L$ be a nontrivially effective line bundle.  Recall that $SD_{2,L}$  is the following strange duality map as in (\ref{sdmt}):
\[SD_{2,L}:H^0(\wtt^L,\ltl)^{\vee}\ra H^0(\ml,\z_L^2(2)).\] 
In this section, we show that under certain conditions $SD_{2,L}$ is an isomorphism (see Theorem \ref{maino}).  

On $\ml$ and $\wtt^L$ we have the following two exact sequences respectively.
\begin{equation}\label{zes}0\ra\z_L(2)\ra\z_L^2(2)\ra\z_L^2(2)|_{D_{\z_L}}\ra 0;\end{equation}
\begin{equation}\label{ses}0\ra \lambda_2(L\otimes K_X)\ra\lambda_2(L)\ra\lambda_2(L)|_{\ts}\ra0.\end{equation}
Notice that $\wtt^{L\otimes K_X}=\wtt^L$ and (\ref{ses}) is because of Corollary \ref{store}.  
\begin{lemma}
By taking the global sections of (\ref{zes}) and the dual of global sections of (\ref{ses}), we have the following commutative diagram
\begin{equation}\label{sddim}\xymatrix{&H^0(\ts,\lambda_2(L)|_{\ts})^{\vee}\ar[r]^{~~~~g_2^{\vee}}\ar[d]_{\alpha_{\ts}}& H^0(\lambda_2(L))^{\vee}\ar[r]^{f_2^{\vee}~~~~}
\ar[d]^{SD_{2,L}}& H^0(\lambda_2(L\otimes K_X))^{\vee}\ar[r]\ar[d]^{\beta_D} &0\\
0\ar[r]& H^0(\z_L(2))\ar[r]_{f_L} & H^0(\z^2_L(2))\ar[r]_{g_L~~~~~~} &H^0(D_{\z_L},\z^2(2)|_{D_{\z_L}})}.
\end{equation}
\end{lemma}
\begin{proof}We only need to show that $g_L\circ SD_{2,L}\circ g_2^{\vee}=0$.  By the definition of $SD_{2,L}$, it is enough to show that the section $\sigma_{c_2^2,L}$ defined in 
Proposition \ref{glob} is identically zero on $\ts\times D_{\z_L}$.  Easy to see that $H^0((\mi_x\oplus\mi_y)\otimes\mg)\neq0$ for all $\mg\in\ml$ such that $H^0(\mg)\neq 0$, 
hence $\ts\times D_{\z_L}\subset\mathtt{D}_{c_2^2,u_L}$ and $\sigma_{c_2^2,L}$ is identically zero on $\ts\times D_{\z_L}$.  The lemma is proved.  
\end{proof}
\subsection{On the map $\alpha_{\ts}$.} \label{alp}~~

We introduce the following condition.
\begin{cond}[$\ca$]The strange duality map 
\begin{equation}\label{rank1}SD_{c^1_2,u_L}:H^0(W(1,0,n),\lambda_{c^1_2}(L))^{\vee}\ra H^0(\ml,\z_L(2))\end{equation}
is an isomorphism.
\end{cond}
\begin{rem}\label{rone}For any $n\geq 1$, $W(1,0,n)\cong X^{[n]}$ and $\lambda_{c^1_n}(L)\cong L_{(n)}$.  
It is well-known that $H^0(X^{[n]},L_{(n)})=S^nH^0(X,L)$ for all $n$ and $L$(see Lemma 5.1 in \cite{EGL}).  
Therefore $\ca$ implies $ H^0(\ls,\mo_{\ls}(2))\cong H^0(\ls,\pi_{*}\z_L\otimes\mo_{\ls}(2))$, in particular we have $h^0(\ml,\z_L)=h^0(\ls,\pi_{*}\z_L)=1$ 
and $D_{\z_L}$ is the unique divisor associated to $\z_L$.
\end{rem}
\begin{lemma}\label{alpha}If $\ca$ is satisfied, then the map $\alpha_{\ts}$ in (\ref{sddim}) is an isomorphism.  In particular, $g_2^{\vee}$ is injective.
\end{lemma}
\begin{proof}By Lemma \ref{slow} we have a surjective map 
\[\varphi^{*\vee}:H^0(X^{(2)},L_2)^{\vee}\twoheadrightarrow H^0(\ts,\lambda_2(L)|_{\ts})^{\vee}.\]
By Proposition 1.2 in \cite{EGL}, we have $HC_2^{*\vee}:H^0(X^{[2]},L_2)^{\vee}\xrightarrow{\cong}H^0(X^{(2)},L_2)^{\vee}$ where $HC_2:X^{[2]}\ra X^{(2)}$ is the Hilbert-Chow morphism.

To prove the lemma, by $\ca$ it is enough to show $\alpha_{\ts}\circ \varphi^{*\vee}\circ HC_2^{*\vee}=SD_{c^1_2,u_L}$ or equivalently 
$SD_{2,L}\circ g_2^{\vee}\circ\varphi^{*\vee}\circ HC_2^{*\vee}=f_L\circ SD_{c^1_2,u_L}$.

We have a Cartesian diagram
\begin{equation}\label{HC}\xymatrix{\widehat{X^2}\ar[r]^{\widehat{HC}}\ar[d]_{\widehat{\mu}}&X^2\ar[d]^{\mu}\\ X^{[2]}\ar[r]_{HC}& X^{(2)}},
\end{equation}
where $\mu$ is a $\ks_2$-quotient and $\widehat{X^2}$ is the blow-up of $X^2$ along the diagonal $\Delta$.  Then we only need to show 
\begin{equation}\label{compare}\widehat{SD}_L:=SD_{2,L}\circ g_2^{\vee}\circ\varphi^{*\vee}\circ HC_2^{*\vee}\circ \widehat{\mu}^{*\vee}=f_L\circ SD_{c^1_2,u_L}\circ\widehat{\mu}^{*\vee}=:\widehat{SD}_R.\end{equation}

There are two flat families on $X\times\widehat{X^2}$ of sheaves of class $c_2^2$:  $\F^1:=\widehat{HC}_X^{*}(pr^{*}_{1,2}\mi_{\Delta}\oplus pr^{*}_{1,3}\mi_{\Delta})$ and $\F^2:=\widehat{\mu}_X^{*}\mathscr{I}_2\oplus q^{*}\mo_X$, where $\widehat{HC}_X:=Id_X\times\widehat{HC}: X\times\widehat{X^2}\ra X^3$, $\widehat{\mu}_X:=Id_X\times\widehat{\mu}$, $q:X\times\widehat{X^2}\ra X$ and $\mathscr{I}_2$ is the universal ideal sheaf on $X\times X^{[2]}$. 

$\F^i$ induces a section $\sigma_i$ of $\widehat{\mu}^{*}\lambda_{c^1_n}(L)\boxtimes \lambda_{L}(c_2^2)\cong \widehat{\mu}^{*}L_n\boxtimes \z_L^2(2)$ on $\widehat{X^2}\times\ml.$
The zero set of $\sigma_i$ is $\mathtt{D}_i:=\{(\underline{x},\mg)|H^0(\F^i_{\underline{x}}\otimes\mg)\neq 0\}$.  
By the definition of $SD_{\crn,u_L}$, we see that $\widehat{SD}_L$ is defined by the global section $\sigma_1$.  On the other hand, the map $f_L$ is defined by multiplying an element in $H^0(\z_L)$ defininig the divisor $D_{\z_L}$.  Therefore $\widehat{SD}_R$ is defined by the global section $\sigma_2$.  Hence to show (\ref{compare}), we only need to show $\mathtt{D}_i$ coincide as divisors for $i=1,2$.

Let $\mc\subset X\times\ls$ be the universal curve.  Then $\mc$ is a divisor in $X\times\ls$.  $p_{i,\ls}:=p_i\times Id_{\ls}:X^2\times\ls\ra X\times\ls$ with $p_i$ the projection to the i-th factor. 
Denote by $p_M:\widehat{X^2}\times\ml\ra\ml$ the projection to $\ml$.  Then easy to see that $\mathtt{D}_1=\mathtt{D}_2=2p_{M}^{*}D_{\z_L}+\widehat{HC}^{*}p_{1,\ls}^{*}\mc+\widehat{HC}^{*}p_{2,\ls}^{*}\mc$.  Hence the lemma.   
\end{proof}
\begin{coro}\label{empty}If $\ca$ is satisfied and moreover $D_{\z_L}=\emptyset$ and $H^0(L\otimes K_X^{\otimes n})=0$ for all $n\geq 1$, then the map $SD_{2,L}$ is an isomorphism.
\end{coro}
\begin{proof}By Lemma \ref{alpha}, we only need to show that $H^0(\lambda_2(L\otimes K_X))=0$.  But $H^0(\lambda_2(L\otimes K_X^{\otimes n})|_{\ts})=0$ since $H^0(L\otimes K_X^{\otimes n})=0$ for all $n\geq 1$.  Hence $H^0(\lambda_2(L\otimes K_X))\cong H^0(\lambda_2(L\otimes K_X^{\otimes n}))$ for all $n\geq 1$ and hence $H^0(\lambda_2(L\otimes K_X))=0$ because $\lambda_2(K_X^{-1})$ is effective. 
\end{proof}

\begin{rem}\label{gzero}Assume $K_X^{-1}$ is effective, then for any curve $C\in|K_X^{-1}|$, either $\mo_C$ is semistable or $C$ contains an integral subscheme with genus $>1$.  Therefore we have $D_{\z_L}=\emptyset\Rightarrow H^0(L\otimes K_X^{\otimes n})=0$ for all $n\geq 1$.  This is because otherwise there must be a semistable sheaf of class $u_L$ 
having nonzero global sections. 

Moreover by Proposition 4.1.1 and Corollary 4.3.2 in \cite{Yuan1}, we see that if every curve in $\ls$ does not contain any 1-dimensional subscheme with positive genus and $K_X^{-1}$ is effective, then Corollary \ref{empty} applies and the strange duality map $SD_{2,L}$ is an isomorphism.   
\end{rem}
We have a useful lemma as follows.
\begin{lemma}\label{dnempty}If $D_{\z_L}\neq \emptyset$, then $L\otimes K_X$ is effective.
\end{lemma}
\begin{proof}Let $\mf\in D_{\z_L}$, then $\Ext^1(\mf,K_X)\cong H^1(\mf)^{\vee}\neq 0$.  Hence there is a non split extension
\[0\ra K_X\ra \widetilde{I}\ra \mf\ra0.\]

If for every proper quotient $\mf\twoheadrightarrow \mf''$ (i.e. $\mf\not\cong \mf''$) we have $h^1(\mf'')=0$, then $\widetilde{I}$ has to be torsion-free and hence isomorphism to $\mi_Z(L\otimes K_X)$ with $Z$ a 0-dimensional subscheme of $X$.  On the other hand $h^0(\widetilde{I})=h^0(\mf)\neq 0$, therefore $H^0(L\otimes K_X)\neq 0$.

If there is a proper quotient $\mf_1$ of $\mf$ such that $h^1(\mf_1)\neq 0$, then we can assume that for every proper quotient $\mf''_{1}$ of $\mf_1$ we have $h^1(\mf''_1)=0$.  Denote by $L_1$ the determinant of $\mf_1$, then by previous argument $H^0(L_1\otimes K_X)\neq 0$ and hence $H^0(L\otimes K_X)\neq 0$ because $L\otimes L_1^{-1}$ is effective.
\end{proof}
\subsection{On the map $\beta_D$.}\label{ssbeta}~~ 

In this subsection we assume $D_{\z_L}\neq \emptyset$, then by Lemma \ref{dnempty} $L\otimes K_X$ is effective.  
We want to prove that under certain conditions the map $\beta_D$ is an isomorphism.  The main technique and notations are analogous to \cite{Yuan5}.  

Let $\ell:=L.(L+K_X)/2=\chi(L\otimes K_X)-1$ and $H_{\ell}$ be the Hilbert scheme of $\ell$-points on $X$ which also parametrizes all ideal sheaves $\mi_Z$ with colength $\ell$, i.e. $len(Z)=\ell$. 
If $\ell=0$, we say $H_0$ is a simple point corresponding to the structure sheaf $\mo_X$.  Denote by $\I_{\ell}$ the universal ideal sheaf over $X\times H_{\ell}$. 

From now on by abuse of notation, we always denote by $p$ the projection $X\times M\ra M$ and $q$ the projection $X\times M\ra X$ for any moduli space $M$.  If we have $Y_1\times\cdots\times Y_n$ with $n\geq2$, denote by $p_{ij}~(i<j)$ the projection to $Y_i\times Y_j$. 

Let $Q_1:=Quot_{X\times H_{\ell}/H_{\ell}}(\I_{\ell}\otimes q^{*}(L\otimes K_X),u_L)$ and $Q_2:=Quot_{X\times H_{\ell}/H_{\ell}}(\I_{\ell}\otimes q^{*}(L\otimes K_X),u_{L\otimes K_X})$ be the two relative Quot-schemes over $H_{\ell}$ parametrizing quotients of class $u_L$ and $u_{L\otimes K_X}$ respectively.  Let $\rho_i:Q_i\ra H_{\ell}$ be the projection.  Each point $[f_1:\mi_Z(L\otimes K_X)\twoheadrightarrow \mf_L]\in Q_1$ ($[f_2:\mi_Z(L\otimes K_X)\twoheadrightarrow \mf_{L\otimes K_X}]\in Q_2$, resp.) over $\mi_Z\in H_{\ell}$ must have the kernel $K_X$ ($\mo_{X}$, resp.).  

Since $L\otimes K_X$ is effective and $X$ is rational, $H^2(L\otimes K_X)=0$.  Hence $h^0(L\otimes K_X)\geq \chi(L\otimes K_X)$.  Therefore, for any ideal sheaf $\mi_Z$ with colenght $\ell$, we have $h^0(\mi_Z(L\otimes K_X))\geq 1$ and hence $\rho_2$ is always surjective.  If moreover $L.K_X\leq0$, then $H^0(\mi_Z(L))\neq 0$ and $\rho_1$ is also surjective.  

We write down the following two exact sequences.
\begin{equation}\label{dd}0\ra K_X\ra \mi_Z(L\otimes K_X)\ra \mf_L\ra0;
\end{equation}
\begin{equation}\label{ddt}0\ra\mo_{X}\ra \mi_Z(L\otimes K_X)\ra \mf_{L\otimes K_X}\ra 0.
\end{equation}
Notice that if $\mf_L$ (resp. $\mf_{L\otimes K_X}$) is semi-stable, then (the class of) $\mf_L$ (resp. (the class of) $\mf_{L\otimes K_X}$) is contained in $D_{\z_L}$ (resp. $M(L\otimes K_X,0)$).  

Let
$$D_{\z_L}^o:=\left\{\mf_L\in D_{\z_L}\left|\begin{array}{l}h^1(\mf_L)=1,~ h^1(\mf_L(-K_X))=0,\\ 
and~Supp(\mf_L)~is~integral.
\end{array}\right.\right\},$$
$$Q_1^o:=\left\{[f_1:\mi_Z(L\otimes K_X)\twoheadrightarrow \mf_L]\in Q_1\left|\begin{array}{l}h^1(\mf_L)=1,~ h^1(\mf_L(-K_X))=0,\\ 
and~Supp(\mf_L)~is~integral.
\end{array}\right.\right\},$$
$$M(L\otimes K_X,0)^o:=\left\{\mf_{L\otimes K_X}\in\mlk\left|\begin{array}{l} 
h^0(\mf_{L\otimes K_X}(K_X))=0,~and\\
Supp(\mf_{L\otimes K_X})~is~integral.
\end{array}\right.\right\},$$
$$Q_2^o:=\left\{[f_2:\mi_Z(L\otimes K_X)\twoheadrightarrow \mf_{L\otimes K_X}]\in Q_2\left|\begin{array}{l}
h^0(\mf_{L\otimes K_X}(K_X))=0,~and\\ 
Supp(\mf_{L\otimes K_X})~is~integral.
\end{array}\right.\right\}.$$

Let  $\mg_r^r$ with $r\geq 1$ be a locally free sheaf of class $c_r^r$ on $X$.  We define a line bundle $\cl^r:=(det(R^{\bullet}p_{*}(\I_{\ell}\otimes q^{*}\mg_r^r(L\otimes K_X))))^{\vee}$ over $H_{\ell}$.  
Then we have the following lemma.
\begin{lemma}\label{change}There are classifying maps $g_1: Q_1^o\ra D_{\z_L}^o$ and $g_2: Q_2^o\ra \mlk^o$, where $g_1$ is an isomorphism and $g_2$ is a projective bundle.
Moreover 
 $g_1^{*}\z^r_L(r)\cong \rho_1^{*}\cl^r|_{Q_1^o}$ and $g_2^{*}\z^r_{L\otimes K_X}(r)\cong\rho_2^{*}\cl^r|_{Q_2^o}$.
\end{lemma}
\begin{proof}The proof is analogous to \cite{Yuan5}.  See Lemma 4.8, Equation (4.9), (4.10), (4.12) and (4.14) in \cite{Yuan5}.
\end{proof}
Let $H_{\ell}^o:=\rho_1(Q_1^o)\cup\rho_2(Q_2^o)$.  We introduce some conditions as follows.  
\begin{cond}[$\cb$]
\begin{enumerate}
\item[$(1)$] $D_{\z_L}^o$ is dense open in $D_{\z_L}$;
\item[$(2)$]
$M(L\otimes K_X,0)$ is of pure dimension and satisfies the ``condition $S_2$ of Serre", and the complement of $M(L\otimes K_X,0)^o$ is of codimension $\geq 2$; %
\item[$(3)$] $(\rho_1)_*\mo_{Q_1^o}\cong\mo_{H^o_{\ell}}$;
\item[$(4)$] $Q_2^o$ is nonempty and dense open in $\rho_2^{-1}(\rho_2(Q_2^o))$.
\end{enumerate}
\end{cond}
\begin{rem}\label{serre}We say a scheme $Y$ satisfies ``condition $S_2$ of Serre" if $\forall ~y\in Y$ the local ring $\mo_y$ has the property that for every prime ideal $\mathfrak{p}\subset\mo_y$ of height $\geq2$, we have $\textbf{depth} ~\mo_{y,\mathfrak{p}}\geq 2$ (also see Ch II Theorem 8.22A in \cite{Ha}).  $\cb$-(2) implies that for every line bundle $\mathcal{H}$ over $\mlk$, the restriction map $H^0(\mlk,\mathcal{H})\hookrightarrow H^0(\mlk^o,\mathcal{H})$ is an isomorphism.
\end{rem}
\begin{lemma}\label{injm}If $\cb$ is satisfied, 
then we have an injective map for all $r>0$
\[j_r: H^0(D_{\z_L},\z_L^r(r)|_{D_{\z_L}})\hookrightarrow H^0(M(L\otimes K_X,0),\z^r_{L\otimes K_X}(r)).\]
Moreover, $j_2\circ\beta_D=SD_{2,L\otimes K_X}$.\end{lemma}
\begin{proof}
By $\cb$-(1) we have an injection 
\begin{equation}\label{map1}H^0(D_{\z_L},\z^r_L(r)|_{D_{\z_L}})\hookrightarrow H^0(D_{\z_L}^o,\z^r_L(r)|_{D^o_{\z_L}}).\end{equation}  
By Lemma \ref{change} and $\cb$-(3) we have 
\begin{equation}\label{map2}
H^0(D^o_{\z_L},\z^r_L(r)|_{D^o_{\z_L}})\xrightarrow{\cong} H^0(Q_1^o,\rho_1^{*}\cl^r|_{Q_1^o})\xrightarrow{\cong} H^0(H^o_{\ell},\cl^r|_{H_{\ell}^o}).\end{equation}

On the other hand $\rho_2$ is projective and surjective, hence there is a natural injection $\mo_{H_{\ell}}\hookrightarrow (\rho_2)_*\mo_{Q_2}$.  
Hence by $\cb$-(4) we have the following injections
\begin{equation}\label{map3}H^0(H^o_{\ell},\cl^r|_{H^o_{\ell}})\hookrightarrow H^0(\rho_2(Q_2^o),\cl^r|_{\rho_2(Q_2^o)})\hookrightarrow 
H^0(\rho_2^{-1}(\rho_2(Q_2^o)),\rho_2^{*}\cl^r)\hookrightarrow H^0(Q_2^o,\rho_2^{*}\cl^r)\end{equation}
Finally by Lemma \ref{change} and $\cb$-(2) we have 
\begin{equation}\label{map4}H^0(Q_2^o,\rho_2^{*}\cl^r)\xrightarrow{\cong}H^0(\mlk^o,\z^r_{L\otimes K_X}(r))\xrightarrow{\cong}H^0(\mlk,\z^r_{L\otimes K_X}(r)).
\end{equation}

The map $j_r$ is obtained by composing all the maps successively in (\ref{map1}), (\ref{map2}), (\ref{map3}) and (\ref{map4}).

Now we prove $j_r\circ \beta_D=SD_{2,L\otimes K_X}$.  Notice that $\chi(\me\otimes \mi_Z(L\otimes K_X))=h^2(\me\otimes \mi_Z(L\otimes K_X))=0$ for all $\me\in\wtt $ and $\mi_Z\in H_{\ell}$.  We then have a determinant line bundle $\lambda_2(\ell)$ (resp. $\lambda_{H_{\ell}}(c_2^2)$ ) over $\wtt^L$ (resp. $H_{\ell}$)  associated to $[\mi_Z(L\otimes K_X)]$ with $\mi_Z\in H_{\ell}$ (resp. $[\me]$ with $\me\in\wtt$).  Obviously $\lambda_{H_{\ell}}(c_2^2)=\cl^2$.  Moreover  there is a section $\sigma_{2,\ell}$ of $H^0(\wtt^L\times H_{\ell},\lambda_2(\ell)\boxtimes\cl^2)$ vanishing at the points $(\me,\mi_Z)$ such that $H^0(\me\otimes\mi_Z(L\otimes K))\neq 0$.  By (\ref{dd}), $\lambda_2(L)\cong \lambda_2(\ell)\otimes \lambda_2([K_X])^{-1}\cong \lambda_2(\ell)\otimes\lambda_2(K_X^{-1})$.  Hence $\lambda_2(\ell)\cong\lambda_2(L\otimes K_X)$.

The section $\sigma_{2,\ell}$ induces a morphism $H^0(\wtt^L,\lambda_2(L\otimes K_X))^{\vee}\xrightarrow{SD_{2,\ell}} H^0(H_{\ell},\cl^2)$.  Composing $SD_{2,\ell}$ with the inclusion $H^0(H_{\ell},\cl^2)\hookrightarrow H^0(H_{\ell}^o,\cl^2)$, we get 
$H^0(\wtt^L,\lambda_2(L\otimes K_X))^{\vee}\xrightarrow{SD^o_{2,\ell}} H^0(H^o_{\ell},\cl^2)$.  Composing maps in (\ref{map3}) and (\ref{map4}) and we get 
$H^0(H_{\ell}^o,\cl^2)\xrightarrow{(g_2)_{*}\circ\rho^{*}_2} H^0(\mlk,\z^2_{L\otimes K_X}(2))$. 

We first show that the following diagram commutes.
\begin{equation}\label{lkcom}\xymatrix{H^0(\wtt^L,\lambda_2(L\otimes K_X))^{\vee}\ar[r]^{~~~~~~~SD^o_{2,\ell}}\ar[rd]_{SD_{2,L\otimes K_X}}& H^0(H^o_{\ell},\cl^2)\ar[d]^{(g_2)_{*}\circ\rho^{*}_2}\\
&H^0(\mlk,\z^2_{L\otimes K_X}(2)).}
\end{equation}

Recall that on $X\times Q^o_2$ there is an exact sequence
\[0\ra \R_2\ra(id_X\times\rho_2^*)\I_{\ell}\otimes q^*(L\otimes K_X)\ra\F_{L\otimes K_X}\ra 0,\]
where $\I_{\ell}$ is the universal sheaf over $X\times H_{\ell}$ and $\R_2=p^{*}\mr_2$ with $\mr_2$ a line bundle over $Q^o_2$.  For simplicity let $\widetilde{\I_2}:=(id_X\times\rho_2^*)\I_{\ell}\otimes q^*(L\otimes K_X)$.  

Recall the good $PGL(V)$-quotient $\rho:\Omega_2\ra \wtt$ such that there is a universal sheaf $\E$ over $X\times \Omega_2$.  
Let $\Omega_2^L:=\rho^{-1}(\wtt^L)$.  The map $H^0(\Omega^L_2,\rho^{*}\lambda_2(L\otimes K_X))^{\vee}\xrightarrow{~~\rho^{*\vee}}H^0(\wtt^L,\lambda_2(L\otimes K_X))^{\vee}$ is surjective and hence to show that (\ref{lkcom}) commutes it suffices to show 
\begin{equation}\label{qlkc}SD_{2,L\otimes K_X}\circ \rho^{*\vee}=(g_2)_{*}\circ \rho_2^*\circ SD_{2,\ell}\circ\rho^{*\vee}.\end{equation} 

Over $X\times \Omega^L_2\times Q^o_2$ we have 
\[0\ra p_{12}^*\E\otimes p_{13}^*\R_2\ra p_{12}^*\E\otimes p_{13}^*\widetilde{\I_2}\ra p_{12}^*\E\otimes p_{13}^*\F_{L\otimes K_X}\ra 0.\]
By Lemma 2.1.20 in \cite{HL}, we have the following commutative diagram
\begin{equation}\label{lkeq}\xymatrix@R=0.6cm@C=0.5cm{&&0&0&0&\\&0\ar[r]&p_{12}^*\E\otimes p_{13}^*\R_2\ar[u]\ar[r]&p_{12}^*\E\otimes p_{13}^*\widetilde{\I_2}\ar[r]\ar[u]& p_{12}^*\E\otimes p_{13}^*\F_{L\otimes K_X}\ar[r]\ar[u]&0\\
0\ar[r]&\C_2\ar[r]&\B'_2\ar[u]\ar[r]&\A_2\ar[u]\ar[r]&p_{12}^*\E\otimes p_{13}^*\F_{L\otimes K_X}\ar[r]\ar[u]^{=}&0\\
&&\B_2\ar[u]\ar[r]^{=}&\B_2\ar[u]&&\\&&0\ar[u]&\C_2\ar[u]&&\\&&&0\ar[u]&&,}\end{equation}
where $\A_2$, $\B'_2$, $\B_2$ and $\C_2$ are locally free such that $R^ip_{*}(\cdot)=0$ for all $i<2$ and $R^2p_{*}(\cdot)$ locally free over $\Omega^L_2\times Q^o_2$.  We have the following commutative diagram
\begin{equation}\label{lkpd}\xymatrix{R^2p_*\C_2\ar[r]&R^2p_{*}\B'_2\ar[r]^{\nu'_2}& R^2p_{*}\A_2\\ R^2p_*\C_2\ar[u]^{=}\ar[r]&R^2p_{*}\B_2\ar[u]^{\eta_2}\ar[r]^{\nu_2}& R^2p_{*}\A_2\ar[u]_{=}.}\end{equation}
$\nu'_2$ and $\nu_2$ are surjective because $H^2(\mf_{L\otimes K_X}\otimes\me)=H^2(\mi_Z(L\otimes K_X)\otimes\me)=0$ for every $[\mi_Z(L\otimes K_X)\twoheadrightarrow\mf_{L\otimes K_X}]\in Q^o_2$ and $\me\in \Omega^L_2$.  $\eta_2$ is an isomorphism because $\R_2$ is a pullback of a line bundle on $Q^o_2$ and $H^1(\me)=H^2(\me)=0$ for all $\me\in\Omega^L_2$.  Denote by $\mk_2$ and $\mk'_2$ the kernels of $\nu_2$ and $\nu'_2$ respectively.  Then we have
\begin{equation}\label{lkd}\xymatrix{R^2p_*\C_2\ar[r]^{~~~\xi_{L\otimes K_X}}&\mk'_2\\ R^2p_*\C_2\ar[u]^{=}\ar[r]_{~~~\xi_{\ell}^2}&\mk_2\ar[u]^{\cong}_{\eta_2}.}\end{equation}
The section $det(\xi_{L\otimes K_X})$ induces the map $g_2^*\circ SD_{2,L\otimes K_X}\circ \rho^{*\vee}$ while the section $det(\xi_{\ell}^2)$ induces the map $\rho_2^*\circ SD^o_{2,\ell}\circ \rho^{*\vee}$.
By (\ref{lkd}) we have $det(\xi_{L\otimes K_X})=det(\eta_2)\cdot det(\xi_{\ell}^2)$ and hence $det(\xi_{L\otimes K_X})$ and $det(\xi_{\ell}^2)$ are the same section up to scalars since $\eta_2$ is an isomorphism.  Hence
\begin{equation}\label{qlkv}g_2^*\circ SD_{2,L\otimes K_X}\circ \rho^{*\vee}=\rho_2^*\circ SD^o_{2,\ell}\circ \rho^{*\vee}.\end{equation} 

(\ref{qlkv}) implies (\ref{qlkc}) because $g_2$ is a projective bundle and the map $H^0(Q_2^o,g_2^*\z_{L\otimes K_X}^r(r))\xrightarrow{(g_2)_*} H^0(\mlk^o,\z_{L\times K_X}^r(r))$ is an isomorphism with inverse map $g_2^*$.  

Now we have that (\ref{lkcom}) commutes.  To show $j_r\circ \beta_D=SD_{2,L\otimes K_X}$, it suffices to show that the following diagram commutes.
\begin{equation}\label{lcom}\xymatrix@C=1.5cm{H^0(\Omega^L_2,\lambda_2(L\otimes K_X))^{\vee}\ar[r]^{~~~~~~~SD^o_{2,\ell}\circ \rho^{*\vee}}& H^0(H^o_e,\cl^2)\\
H^0(\Omega^L_2,\lambda_2(L))^{\vee}\ar[u]^{f_{2,\Omega}^{\vee}}\ar[r]_{g_L\circ SD_{2,L}\circ \rho^{*\vee}~~}&H^0(D_{\z_L},\z^2_{L}(2)|_{D_{\z_L}})\ar[u]_{(\rho_1)_{*}\circ g_1^*}.}
\end{equation}
In other words, it suffices to show
\begin{equation}\label{qlc}(\rho_1)_{*}\circ g_1^*\circ g_L\circ SD_{2,L}\circ \rho^{*\vee}=SD^o_{2,\ell}\circ \rho^{*\vee}\circ f_{2,\Omega}^{\vee}.
\end{equation}

Recall that on $X\times Q^o_1$ there is an exact sequence
\[0\ra \R_1\ra(id_X\times\rho_1^*)\I_{\ell}\otimes q^*(L\otimes K_X)\ra\F_{L}\ra 0,\]
where $\I_{\ell}$ is the universal sheaf over $X\times H_{\ell}$ and $\R_1=p^{*}\mr_1\otimes q^*K_X$ with $\mr_1$ a line bundle (actually the relative tautological bundle $\mo_{\rho_1}(-1)$) over $Q^o_1$.  Let $\widetilde{\I_1}:=(id_X\times\rho_1^*)\I_{\ell}\otimes q^*(L\otimes K_X)$.  

Over $X\times \Omega^L_2\times Q^o_1$ we have 
\[0\ra p_{12}^*\E\otimes p_{13}^*\R_1\ra p_{12}^*\E\otimes p_{13}^*\widetilde{\I_1}\ra p_{12}^*\E\otimes p_{13}^*\F_{L}\ra 0.\]
Analogously, we have the following commutative diagram
\begin{equation}\label{leq}\xymatrix@R=0.6cm@C=0.5cm{&&0&0&0&\\&0\ar[r]&p_{12}^*\E\otimes p_{13}^*\R_1\ar[u]\ar[r]&p_{12}^*\E\otimes p_{13}^*\widetilde{\I_1}\ar[r]\ar[u]& p_{12}^*\E\otimes p_{13}^*\F_{L}\ar[r]\ar[u]&0\\
0\ar[r]&\C_1\ar[r]&\B'_1\ar[u]\ar[r]&\A_1\ar[u]\ar[r]&p_{12}^*\E\otimes p_{13}^*\F_{L}\ar[r]\ar[u]^{=}&0\\
&&\B_1\ar[u]\ar[r]^{=}&\B_1\ar[u]&&\\&&0\ar[u]&\C_1\ar[u]&&\\&&&0\ar[u]&&,}\end{equation}
where $\A_1$, $\B'_1$, $\B_1$ and $\C_1$ are locally free such that $R^ip_{*}(\cdot)=0$ for all $i<2$ and $R^2p_{*}(\cdot)$ locally free over $\Omega^L_2\times Q^o_1$.  We have the following commutative diagram
\begin{equation}\label{lpd}\xymatrix{R^2p_*\C_1\ar[r]&R^2p_{*}\B'_1\ar[r]^{\nu'_1}& R^2p_{*}\A_1\\ R^2p_*\C_1\ar[u]^{=}\ar[r]&R^2p_{*}\B_1\ar[u]^{\eta_1}\ar[r]^{\nu_1}& R^2p_{*}\A_1\ar[u]_{=}.}\end{equation}
$\nu'_1$ and $\nu_1$ are surjective because $H^2(\mf_{L}\otimes\me)=H^2(\mi_Z(L\otimes K_X)\otimes\me)=0$ for every $[\mi_Z(L\otimes K_X)\twoheadrightarrow\mf_{L}]\in Q^o_1$ and $\me\in \Omega^L_2$.  $\eta_1$ is a morphism between two vector bundles with same rank with cokernel $R^2p_{*}(p_{12}^*\E\times p_{13}^*\R_1)$.  Since $\R_1\cong p^{*}\mr_1\otimes q^*K_X$ with $\mr_1$ a line bundle over $Q_1^o$, $det(\eta_1)$ is the pullback to $\Omega^L_2\times Q_1^o$ of the section of $\lambda_2([K_X]^{-1})\cong \lambda_2(K_X^{-1})$ defining the subscheme $\widetilde{\ts}$.  

Denote by $\mk_1$ and $\mk'_1$ the kernels of $\nu_1$ and $\nu'_1$ respectively.  Then we have
\begin{equation}\label{ld}\xymatrix{R^2p_*\C_1\ar[r]^{~~~\xi_{L}}&\mk'_1\\ R^2p_*\C_1\ar[u]^{=}\ar[r]_{~~~\xi_{\ell}^1}&\mk_1\ar[u]_{\eta_1}.}\end{equation}
The section $det(\xi_{L})$ induces the map $g_1^*\circ g_L\circ SD_{2,L}\circ \rho^{*\vee}$, the section $det(\xi_{\ell}^1)$ induces the map $\rho_1^*\circ SD^o_{2,\ell}\circ \rho^{*\vee}$ and multiplying the section $det(\eta_1)$ induces the map $f_{2,\Omega}^{\vee}$.
By (\ref{ld}) we have $det(\xi_{L})=det(\eta_1)\cdot det(\xi_{\ell}^1)$ and hence 
\begin{equation}\label{qlv}g_1^*\circ g_L\circ SD_{2,L}\circ \rho^{*\vee}=\rho_1^{*}\circ SD^o_{2,\ell}\circ \rho^{*\vee}\circ f_{2,\Omega}^{\vee}.
\end{equation} 

(\ref{qlv}) implies (\ref{qlc}) because by $\cb$-(3) the map $H^0(Q_1^o,\rho_1^*\cl^r)\xrightarrow{(\rho_1)_*} H^0(H^o_{\ell},\cl^r)$ is an isomorphism with inverse map $\rho_1^*$.  

The lemma is proved.
\end{proof}

Now we want to modify $\cb$.  
Define
$$|L\otimes K_X|':=\left\{C\in|L\otimes K_X|\left|\begin{array}{l} For~every~integral~subscheme~C_1\subset C, \\we~have~ deg(K_X|_{C_1})<0.
\end{array}\right.\right\},$$
$$\mlk':=\left\{\mf_{L\otimes K_X}\in\mlk\left|\begin{array}{l}
h^0(\mf_{L\otimes K_X}(K_X))=0,~and\\ 
Supp(\mf_{L\otimes K_X})\in |L\otimes K_X|'.
\end{array}\right.\right\}.$$
$$Q_2':=\left\{[f_2:\mi_Z(L\otimes K_X)\twoheadrightarrow \mf_{L\otimes K_X}]\in Q_2\left|\begin{array}{l}
\mf_{L\otimes K_X}~is~semistable,\\ h^0(\mf_{L\otimes K_X}(K_X))=0,~and\\ 
Supp(\mf_{L\otimes K_X})\in |L\otimes K_X|'.
\end{array}\right.\right\}.$$

Let $f_M:\Omega_{L\otimes K_X}\ra \mlk
$ be the good $PGL(V_{L\otimes K_X})$-quotient with $V_{L\otimes K_X}$ some vector space and $\Omega_{L\otimes K_X}$ an open subscheme of some Quot-scheme.   Let
$\Omega'_{L\otimes K_X}:= f_M^{-1}(\mlk')$.  Notice that $\Ext^2(\mf_{L\otimes K_X},\mf_{L\otimes K_X})=0$ for $\mf_{L\otimes K_X}$ semistable with $Supp(\mf_{L\otimes K_X})\in |L\otimes K_X|'$.
Hence $\Omega'_{L\otimes K_X}$ is smooth of pure dimension the expected dimension.

Denote by $\Q_{L\otimes K_X}$ the universal quotient over $\Omega_{L\otimes K_X}$.  Analogous to \cite{Yuan5}, define $\mv':=\mathscr{E}xt_p^1(\Q_{L\otimes K_X}|_{\Omega'_{L\otimes K_X}},q^{*}\mo_X)$ which is locally free of rank $-(L+K_X).K_X$ on $\Omega'_{L\otimes K_X}$.  
Let $P'_2\subset\p(\mv')$ parametrizing torsion free extensions of $\Q_{\mathfrak{s}}$ by $\mo_X$ for all $\mathfrak{s}\in \Omega'_{L\otimes K_X}$.  Then the classifying map $f'_{Q_2}: P'_2\ra Q'_2$ is a principal $PGL(V_{L\otimes K_X})$-bundle (see Lemma 4.7 in \cite{Yuan5}).  We have the following commutative diagram
\begin{equation}\label{pie}\xymatrix@C=2cm{P_2'\ar[r]^{\sigma'_2}\ar[d]_{f_{Q_2}'} &\Omega'_{L\otimes K_X}\ar[d]^{f'_M}\\  Q_2'\ar[r]_{g_2'~~~~~~~~~~~} &\mlk' }.\end{equation}   

Let $H_{\ell}':=\rho_1(Q_1^o)\cup\rho_2(Q_2')$.  We define $\cb'$ by keeping $\cb$-(1) and replacing $\cb$-(2), (3) and (4) by $(2'a)$, $(2'b)$, $(3)$ and $(4')$ as follows.
\begin{enumerate}
\item[$(2'a)$]
$M(L\otimes K_X,0)$ is of pure dimension and satisfies the ``condition $S_2$ of Serre", and the complement of $M(L\otimes K_X,0)'$ is of codimension $\geq 2$;  
\item[$(2'b)$] 
The complement of $P_2'$ in $\p(\mv')$ is of codimension $\geq 2$; %
\item[$(3')$] $(\rho_1)_{*}\mo_{Q_1^o}\cong\mo_{H_{\ell}'}$;
\item[$(4')$] $Q'_2$ is nonempty and dense open in $\rho_2^{-1}(\rho_2(Q'_2))$.
\end{enumerate}
\begin{lemma}\label{zero}
If $\cb'$ is satisfied, then there is an injective map for all $r>0$
\[j_r: H^0(D_{\z_L},\z_L^r(r)|_{D_{\z_L}})\hookrightarrow H^0(M(L\otimes K_X,0),\z^r_{L\otimes K_X}(r)),\]
such that $j_2\circ\beta_D=SD_{2,L\otimes K_X}$.\end{lemma}
\begin{proof}The only difference from Lemma \ref{injm} is that the map $g_2'$ is no more a projective bundle.  
However it is enough to prove $(g_2')_{*}\mo_{Q_2'}\cong \mo_{\mlk'}\cong \mo_{\mlk}$.

In (\ref{pie}) we have $f'_{Q_2}$ a principal $PGL(V_{L\otimes K_X})$-bundle and $f'_M$ a good $PGL(V_{L\otimes K_X})$-quotient.  $\sigma'_2$ is $PGL(V_{L\otimes K_X})$-equivariant and descends to the map $g'_2$.  
In order to show $(g'_{2})_{*}\mo_{Q'_2}\cong \mo_{|L\otimes K_X|'}$, we only need to show that $(\sigma'_{2})_{*}\mo_{P'_2}\cong\mo_{\Omega'_{L\otimes K_X}}$.  

We have that $(\sigma_{2})_{*}\mo_{\p(\mv')}\cong\mo_{\Omega'_{L\otimes K_X}}$.  $\Omega'_{L\otimes K_X}$ is smooth of pure dimension.  By $\cb'$-$(2'b)$ the complement of $P'_2$ in $\p(\mv')$ is of codimension $\geq2$ and hence $\jmath_{*}\mo_{P'_2}\cong \mo_{\p(\mv')}$ with $\jmath:P'_2\hookrightarrow \p(\mv')$ the embedding.   On the other hand $\sigma'_2=\sigma_2\circ \jmath$, hence $(\sigma'_{2})_{*}\mo_{P'_2}\cong(\sigma_{2})_{*}( \jmath_{*}\mo_{P'_2})\cong(\sigma_{2})_{*}\mo_{\p(\mv')}\cong\mo_{\Omega'_{L\otimes K_X}}$.  Hence the lemma  
\end{proof} 
Notice that $\cb$-(2) $\Rightarrow$ $\cb'$-$(2'a)$ if $(L+K_X).K_X<0$.  Lemma \ref{injm} and Lemma \ref{zero} imply immediately the following proposition.
\begin{prop}\label{beta}If either $\cb$ or $\cb'$ is satisfied and $SD_{2,L\otimes K_X}$ is an isomorphism, then the map $\beta_D$ in (\ref{sddim}) is an isomorphism.  In particular, $g_L$ is surjective. 
\end{prop}
\begin{rem}\label{lek}
If $L\cong K_X^{-1}$ then $\beta_D$ is an isomorphism as long as $\forall~C\in\ls$, $\mo_C$ is stable (which is equivalent to say that $C$ contains no subcurve with genus $\geq1$) and there is a stable vector bundle $\me\in\wtt$.  This is because in this case $\beta_D$ is a nonzero map between two vector spaces of 1 dimension, hence an isomorphism.  $\beta_D$ is nonzero since $H^0(\me\otimes\mo_C)=H^1(\me\otimes\mo_C)=0$ for all $C\in\ls$ ( also see the proof of Proposition 6.25 in \cite{GY}).
\end{rem}

Combining Lemma \ref{alpha} and Proposition \ref{beta} we have the following theorem.

\begin{thm}\label{maino}Assume $\ca$ and either $\cb$ or $\cb'$ are satisfied, then $SD_{2,L}$ is an isomorphism if $SD_{2,L\otimes K_X}$ is an isomorphism.\end{thm}

\subsection{Application to Hirzebruch surfaces.}\label{app}~~

Theorem \ref{maino} applies to a large number of cases on Hirzebruch surface as stated in the following theorem.
\begin{thm}\label{ruled}Let $X=\Sigma_e~(e\geq0):=\p(\mo_{\p^1}\oplus\mo_{\p^1}(e))$.  Let $F$ be the fiber class and $G$ the section such that $G^2=-e$ over $X$.  
Let $L=aG+bF$.  Then

(1) $\ca$ is fulfilled for $L$ ample or $\min\{a,b\}\leq1$.

(2) If $2\leq\min\{a,b\}\leq 3$, then $\cb'$ is fulfilled for $L$ ample, i.e. $b> ae$ for $e\neq0$; or $a,b>0$ for $e=0$.

(3) If $\min\{a,b\}\geq 4$, then $\cb$ is fulfilled for both $L$ and $L\otimes K_X$ ample, i.e. $b> ae, e>1$; or $b> a+1,e=1$; or $a,b\geq4,e=0$.
\end{thm}

\begin{coro}\label{ok}Let $X$ be a Hirzebruch surface $\Sigma_e$ and $L=aG+bF$.  Then the strange duality map $SD_{2,L}$ in (\ref{sdmt}) is an isomorphism for the following cases.
\begin{enumerate}
\item $\min\{a,b\}\leq1$;
\item $\min\{a,b\}\geq2$, $e\neq 1$, $L$ ample;
\item $\min\{a,b\}\geq2$, $e=1$, $b\geq a+[a/2]$ with $[a/2]$ the integral part of $a/2$.
\end{enumerate}
\end{coro}
\begin{proof}If $\min\{a,b\}\leq1$,  then every curve in $\ls$ contains no subcurve of positive genus and hence done by Corollary \ref{empty} and Remark \ref{zero}.  

If $\min\{a,b\}\geq2$ and $e\neq1$, then $L$ is ample $\Rightarrow$ $L\otimes K_X$ is ample.   Therefore by Theorem \ref{ruled} and Theorem \ref{maino} we can reduce the problem to $L=G+nF$ (or $F+nG$ for $e=0$), or $nF$ (or $mG$ for $e=0$) while by Corollary \ref{empty} and Remark \ref{zero}, $SD_{2,L}$ is an isomorphism in these cases.   

If $\min\{a,b\}\geq2$, $e=1$ and $b\geq a+[a/2]$, then either both $L$ and $L\otimes K_X$ are ample or $L$ ample and $L\otimes K_X=G+F~or~nF$.  Therefore analogously we are done by Theorem \ref{ruled}, Theorem \ref{maino}, Corollary \ref{empty} and Remark \ref{zero}. 

The corollary is proved.
\end{proof}

To prove Theorem \ref{ruled}, the main task is estimating codimension of some schemes.  However we want to use stack language as what we did in \cite{Yuan4} because it makes the argument clearer and simpler.  Therefore, we firstly introduce some stacks as follows, the notations of which are slightly different from \cite{Yuan4}.
\begin{defn}\label{ff}Let $\chi$ and $d$ be two integers.
 
(1) Let $\mm^d(L,\chi)$ be the (Artin) stack parametrizing pure 1-dimensional sheaves $\mf$ on $X$ with determinant $L$, Euler characteristic $\chi(\mf)=\chi$ and satisfying 
either $\mf$ is semistable or  $\forall \mf'\subset \mf$, $\chi(\mf')\leq d$.
 
(2) Let $\mm(L,\chi)$ ($\mm(L,\chi)^s$, resp.) be the substack of $\mm^d(L,\chi)$ parametrizing semistable (stable, resp.) sheaves in $\mm^d(L,\chi)$.

(3) Let $\mm^{int}(L,\chi)$ be the substack of $\mm(L,\chi)^s$ parametrizing sheaves with integral supports.

(4) Let $\mm^{d,R}(L,\chi)$ be the substack of $\mm^d(L,\chi)$ parametrizing sheaves with reducible supports in $\mm^d(L,\chi)$.  
Let $\mm^{R}(L,\chi)=\mm^{d,R}(L,\chi)\cap\mm(L,\chi)^s$.

(5) Let $\mm^{d,N}(L,\chi)$ be the substack of $\mm^d(L,\chi)$ parametrizing sheaves with irreducible and non-reduced supports in $\mm^d(L,\chi)$.
Let $\mm^{N}(L,\chi)=\mm^{d,N}(L,\chi)\cap\mm(L,\chi)^s$.

(6) Let $\mc^d(nL',\chi)$ $(n>1)$ be the substack of $\mm^d(nL',\chi)$ parametrizing sheaves $\mf$ whose supports are of the form $nC$ with $C$ an integral curve in $|L'|$.  $\mc(nL',\chi)=\mc^d(nL',\chi)\cap\mm(L,\chi)^s$.
\end{defn}
\begin{lemma}\label{codim}Let $X=\Sigma_e$ and $L=aG+bF$ ample with $\min\{a,b\}\geq2$. 
Then for all $\chi$ and $d$, $\mm^{int}(L,\chi)$ is smooth of dimension $L^2$,
and the complement of $\mm^{int}(L,\chi)$ inside $\mm^d(L,\chi)$ is of codimension $\geq2$, i.e. of dimension $\leq L^2-2$.
\end{lemma}
\begin{proof}
Since $L.K_X<0$, $\mm^{int}(L,\chi)$ is smooth of dimension $L^2$.
We first estimate the dimension of $\mc^d(nL',\chi)$ ($n>1$).  
Write $L'=a'G+b'F$.  Since $|L'|^{int}\neq\emptyset$, $L'=G$ or $F$; or $b'\geq a'e$, $e>0$; or $a',b'>0$, $e=0$.  

\emph{Claim $\clubsuit$.} $\forall~d,\chi$, $dim~\mc^d(nL',\chi)\leq n^2L'^2-\min\{7,-nL'.K_X-1,(n-1)L'^2\}\leq n^2L'^2$ for $L'$ nef and $dim~\mc^d(nG,\chi)\leq -n^2$ for $e>0$. 

We show Claim $\clubsuit$.  Let 
$$\mt_m(L,\chi):=\{\mf\in \mm(L,\chi) |\exists ~x\in X, such~that~dim_{k(x)}(\mf\otimes k(x))\geq m\},$$
where $k(x)$ is the residue field of $x$.
Take a very ample divisor $H=G+(e+1)F$ on $X$.  If $L'$ is nef, then $(-jH+K_X).L'<0$ for all $j\geq-1$ and hence $ H^1(\mathscr{E}xt^1(\mf,\mf)(jH))\cong\Ext^2(\mf,\mf (jH))\cong  \Hom(\mf,\mf(K_X-jH))^{\vee}=0$ for all $j\geq-1$ and $\mf \in \mc(nL',\chi)$.  Therefore by Castelnuovo-Mumford criterion $\mathscr{E}xt^1(\mf,\mf)$ is globally generated.  Hence by Le Potier's argument in the proof of Lemma 3.2 in \cite{LP1}, $\mc(nL',\chi)\cap\mt_m(nL',\chi)$ is of dimension $\leq n^2L'^2-m^2+2$.  Combining Proposition 4.1 and Theorem 4.16 in \cite{Yuan4}, we have 
\begin{equation}\label{nrdim}dim~\mc(nL',\chi)\leq n^2L'^2-\min\{7,n(n-1)L'^2,-nK_X.L'-1\}.\end{equation}  

Let $\mf\in\mc^d(nL',\chi)\setminus \mc(nL',\chi)$.  Since $\forall~\mf'\subset\mf$, $K_X.c_1(\mf')<0$, the proof of Proposition 2.7 in \cite{Yuan4} applies and $dim~(\mc^d(nL',\chi)\setminus \mc(nL',\chi))\leq n^2L'^2-(n-1)L'^2$.  

Let $e>0$.  For every semistable sheaf $\mf$ with support $nG$, the map $\mf\xrightarrow{\cdot\delta_G}\mf(G)$ is zero because $G^2<0$, where $\delta_G\in H^0(\mo_X(G))$ is a function defining the divisor $G$.  Hence $\mf$ is a sheaf on $G$ and hence a direct sum of $n$ line bundles over $G$.  Thus $dim~\mc(nG,\chi)\leq-n^2$.  Let $\mf$ be unstable with support $G$, then take the Harder-Narasimhan filtration of it as follows.
\[0=\mf_0\subsetneq \mf_1\subsetneq\cdots\subsetneq \mf_k=\mf,\]
with $\mf_i/\mf_{i-1}\cong\mo_G(s_i)^{\oplus n_i}$ such that $s_1>s_2>\cdots >s_{k}$ and $\sum_{i=1}^k n_i=n$.  Then $\text{ext}^2(\mf_i/\mf_{i-1},\mf_{i-1})=\text{hom}(\mf_{i-1},\mf_{i}/\mf_{i-1}(K_X))\leq\sum_{j<i}\text{hom}(\mo_G(s_j)^{\oplus n_j},\mo_G(s_i+(e-2))^{\oplus n_i})\leq \sum_{j<i}(e-2)n_in_j.$
By induction assumption $dim~\mc^d(\tilde{n}G,\chi)\leq-\tilde{n}^2$ for all $\tilde{n}<n$, then by analogous argument to the proof of Proposition 2.7 in \cite{Yuan4} we have 
\[\begin{array}{l}dim~\mc^d(nG,\chi)\\
\leq\displaystyle{\max_{\tiny{\begin{array}{c}n_1,\cdots,n_k>0\\ \sum_{i} n_i=n\end{array}}}}\{-n^2,-(\sum_{i=1}^{k-1}n_i)^2-n_k^2+\text{ext}^1(\mf_k/\mf_{k-1},\mf_{k-1})-\text{hom}(\mf_k/\mf_{k-1},\mf_{k-1})\}\\
=\displaystyle{\max_{\tiny{\begin{array}{c}n_1,\cdots,n_k>0\\ \sum_{i} n_i=n\end{array}}}}\{-n^2,-(\sum_{i=1}^{k-1}n_i)^2-n_k^2-\chi(\mf_i/\mf_{k-1},\mf_{k-1})+\text{ext}^2(\mf_k/\mf_{k-1},\mf_{k-1})\}\\
\leq \displaystyle{\max_{\tiny{\begin{array}{c}n_1,\cdots,n_k>0\\ \sum_{i} n_i=n\end{array}}}}\{-n^2,-(\sum_{i=1}^{k-1}n_i)^2-n_k^2-n_k(\sum_{i=1}^{k-1}n_i)e+(e-2)(\sum_{i=1}^{k-1}n_in_k)\}=-n^2 \end{array}\]
Therefore Claim $\clubsuit$ is proved.

Easy to see $\mm^d(L,\chi)\setminus\mm^{int}(L,\chi)=\mm^{d,R}(L,\chi)\cup\mm^{d,N}(L,\chi)$ and $\mm^{d,N}(L,\chi)=\cup_{nL'=L}\mc^d(nL',\chi)$. 
Claim $\clubsuit$ implies that $\mm^{d,N}(L,\chi)$ is of codimension $\geq 2$ inside $\mm^d(L,\chi)$ for $L$ ample.
Now we only need to show $\mm^{d,R}(L,\chi)$ is of dimension $\leq L^2-2$.

Let $\mg\in\mm^{d,R}(L,\chi)$, then $\mg$ admits a filtration as follows.
\[0=\mg_0\subsetneq \mg_1\subsetneq\cdots\subsetneq \mg_l=\mg,\]
with $\ms_i:=\mg_i/\mg_{i-1}\in\mc^{d_i}(n_iL_i,\chi_i)$ such that $\displaystyle{\sum_{i=1}^l} n_iL_i=L$, $\displaystyle{\sum_{i=1}^l} \chi_i=\chi$ and $\Hom(\ms_i,\ms_j)=\Ext^2(\ms_i,\ms_j)=0,~\forall~i\neq j$.  Hence
$\text{ext}^1(\ms_i,\ms_j)=-\chi(\ms_i,\ms_j)=n_in_j(L_i.L_j)$ for all $i>j$, and $\text{ext}^1(\ms_i,\mg_{i-1})=\displaystyle{\sum_{j<i}}\text{ext}^1(\ms_i,\mg_{i-1})$.
By analogous argument to the proof of Proposition 2.7 in \cite{Yuan4}, we have

\begin{equation}\label{redd}\begin{array}{l}dim~\mm^{d,R}(L,\chi)\leq \displaystyle{\max_{\sum n_iL_i=L}} \{\sum_{i}dim~\mc^{d_i}(n_iL_i,\chi_i)+\sum_{j<i}n_in_j(L_i.L_j)\}\\
\leq\displaystyle{\max_{\tiny{\begin{array}{c}\sum n_iL_i=L-a_0G\\ L_i ~nef, a_0\leq a\end{array}}} }\{\sum_{i}n^2_iL_i^2+\sum_{j<i}n_in_j(L_i.L_j)-a_0^2+a_0G.(L-a_0G)\}
\\ = \displaystyle{\max_{\tiny{\begin{array}{c}\sum n_iL_i=L-a_0G\\ L_i ~nef, a_0\leq a\end{array}}} }\{L^2-\sum_{j<i}n_in_j(L_i.L_j)-a_0^2-a_0G.L\}\\
=L^2-\displaystyle{\min_{\tiny{\begin{array}{c}\sum n_iL_i=L-a_0G\\ L_i ~nef, a_0\leq a\end{array}}} }\{\sum_{j<i}n_in_j(L_i.L_j)+a_0^2+a_0G.L\}
\end{array}\end{equation} 
If $a_0\geq 1$, then $\sum_{j<i}n_in_j(L_i.L_j)+a_0^2+a_0G.L\geq a_0^2+a_0(b-ea)\geq 2$.  If $a_0=0$ or $e=0$, then $\sum_{j<i}n_in_j(L_i.L_j)\geq 2$ since $\min\{a,b\}\geq 2$ and $L_i$ are all nef.  Hence the lemma is proved. 
\end{proof}
\begin{rem}\label{dfo}Let $d$, $\chi$ be two integers.  Claim $\clubsuit$ and (\ref{redd}) also provide an estimate of $dim~\mm^d(L,\chi)$ for all $L$ effective.  We can see that $dim~\mm^d(nG,\chi)=dim~\mc^d(nG,\chi)\leq-n^2$ for $e\neq 0$ and $dim~\mm^d(nF,\chi)\leq 0$. 

Denote by $\ls^{int}$ the open subset of $\ls$ consisting of all integral curves.  If $L$ is nef and big, i.e. $\ls^{int}\neq \emptyset$ and $L\neq F,G$, then $L.K_X<0$ and $dim~\mm^{int}(L,\chi)$ is smooth of dimension $L^2$, and moreover by Claim $\clubsuit$ and (\ref{redd}), $\mm^d(L,\chi)\setminus\mm^{int}(L,\chi)\leq L^2-1$.  Hence $dim~\mm^d(L,\chi)=dim~\mm(L,\chi)^s=L^2$ and $\mm(L,\chi)$ is irreducible of expected dimension.

If $\ls^{int}=\emptyset$, $\min\{a,b\}\geq1$ and $-K_X$ is nef, i.e. $e\leq2$; then $\mm(L,\chi)^s$ is either empty or of smooth of dimension $L^2$.

If $\ls^{int}=\emptyset$ with $\min\{a,b\}\geq1$, then $\mm^d(L,\chi)=\mm^{d,R}(L,\chi)$ and we then have 
\[dim~\mm^d(L,\chi)\leq \displaystyle{\max_{L-a_0G~nef}}\{(L-a_0G)^2+a_0G(L-a_0G)-a_0^2\}.\]
Let $\mf_L$ be stable with $C_{\mf_L}=a_0G+C'_{\mf_L}$ such that $G$ is not a component of $C'_{\mf_L}$, let $\mf^G_L$ be $\mf_L\otimes \mo_{a_0G}$ modulo its torsion.  Hence $\mf^G_L$ is a quotient of $\mf_L$ while $\mf^G_L(-C'_{\mf_L})$ is a subsheaf of $\mf_L$.  Hence by stability of $\mf_L$, $C'_{\mf_L}.G>0$ and $L-a_0G$ must be either ample or $bF$.  Hence
\begin{equation}\label{empty}dim~\mm(L,\chi)^s\leq \displaystyle{\max_{\tiny{\begin{array}{c}L-a_0G~ample\\ or~ a_0=a\end{array}}}}\{(L-a_0G)^2+a_0G(L-a_0G)-a_0^2\}.\end{equation}
\end{rem}
We can choose an atlas $\Omega^d_{L,\chi}\xrightarrow{\psi} \mm^d(L,\chi)$ with $\Omega^d_{L,\chi}$ a subscheme of some Quot-scheme.  We also can ask $\psi^{-1}(\mm(L,\chi))=:\Omega_{L,\chi}\xrightarrow{f_M} M(L,\chi)$ to be a good $PGL(V_{L,\chi})$-quotient with $M(L,\chi)$ the coarse moduli space of semistable sheaves.  Analogously we define 
$\Omega^{s}_{L,\chi}$, $\Omega^{int}_{L,\chi}$, $\Omega^{d,R}_{L,\chi}$, $\Omega^{d,N}_{L,\chi}$ etc.  If $\chi=0$, we write $\Omega^{\bullet}_L$ instead of $\Omega_{L,0}^{\bullet}$.  
Since $\psi$ is smooth, the codimension of $\Omega^{\bullet}_{L,\chi}$ inside $\Omega^d_{L,\chi}$ is the same as $\mm^{\bullet}(L,\chi)$ inside $\mm^d(L,\chi)$.  ``$\bullet$" stands for 
``$int$", ``$d,R$", ``$d,N$" etc.

Let $M^{int}(L,\chi):=\pi^{-1}(\ls^{int})$.  Then $M^{int}(L,\chi)$ is a flat family of (compactified) Jacobians over $\ls^{int}$, hence it is connected.  $\Omega_{L,\chi}^{int}=f_M^{-1}(M^{int}(L,\chi))$ and $\Omega_{L,\chi}^{int}$ is a principal $PGL(V_L)$-bundle over $M^{int}(L,\chi)$ hence also connected.

We have a corollary to Lemma \ref{codim} as follows.
\begin{coro}\label{inno}Let $X=\Sigma_e$ and $L=aG+bF$.

(1) If $\min\{a,b\}\leq1$, then $\ml\cong\ls$ and $\z_L\cong\mo_{\ls}$.

(2) If $\min\{a,b\}\geq2$ and $L$ is nef for $e\neq 1$, ample for $e=1$, then 
$M(L,0)$ is integral and normal; $M(L,0)\setminus \mli$ is of codimension $\geq2$ inside $\ml$;
and the dualizing sheaf of $\ml$ is locally free and isomorphic to $ \pi^{*}\mo_{\ls}(L.K_X)$.
Moreover $\pi_{*}\z_L\cong\mo_{\ls}$ and $R^i\pi_{*}\z_L^r=0$ for all $i,r>0$.  
\end{coro}
\begin{proof}If $\min\{a,b\}\leq1$, then done by Proposition 4.1.1 in \cite{Yuan1}.

Let $L$ be as in (2).  There are nonsingular irreducible curves in $\ls$ and the complement of $\ls^{int}$ in $\ls$ is of codimension $\geq 2$.  Since $L.K_X<0$, $\mli$ is smooth and irreducible of dimension $L^2+1$.
$\Omega_{L}^{int}$ is also smooth, hence irreducible and of expected dimension.

By Lemma \ref{codim}, $\Omega^d_L\setminus\Omega_L^{int}$ is of codimension $\geq 2$ inside $\Omega^d_L$, then $\Omega_L^{int}$ is dense in 
$\Omega_L$, hence then $\Omega_L$ is of expected dimension and by deformation theory $\Omega_L$ is a local complete intersection.  On the other hand, $\Omega_L$
is smooth in codimension 1, hence normal for local complete intersection.  Therefore $\ml$ is integral and normal since $\Omega_L$ is.

To show that $\ml\setminus\mli$ is of codimension $\geq 2$, we only need to show $\ml\setminus\ml^s$ is of codimension $\geq 2$ with $\ml^s$ the open subset consisting of stable sheaves.  By Remark \ref{dfo}, $dim~M(L',0)^s=L'^2+1$ for $L'$ nef and big, $dim~M(F,0)^s=1$, $dim~M(nF,0)^s=0$ for $n>1$, $dim~M(nG,0)=0$ for $e>0$ and finally by (\ref{empty}) 
for $|L'|^{int}=\emptyset$ and $L'\neq nF,mG,$
$$dim~M(L',0)^s\leq\max_{\tiny{\begin{array}{c}L-a_0G~ample\\ or~ a_0=a\end{array}}}\{(L'-a_0G)^2+1-a_0^2+a_0G.(L'-a_0G)\}.$$  
Hence if $e\neq 0$, then
\begin{equation}\label{ssemi}\begin{array}{l}L^2+1-dim~(\ml\setminus\ml^s)=L^2+1-\displaystyle{\max_{\sum L_i=L}} \{\sum_{i}dim~M(L_i,0)^s\}\\
\leq L^2+1-\displaystyle{\max_{\tiny{\begin{array}{c}\sum L_i=L-a_0G\\ L'_i:=L_i-a_iG~nef,
\\ a_i\geq0,~ a_0\leq a\end{array}}} }\{\sum_{i}(L_i-a_iG)^2-a_i^2+a_iG.(L_i-a_iG)+\#\{L_i\}\}
\\ \leq\displaystyle{\min_{\tiny{\begin{array}{c}\sum L_i=L-a_0G\\ L'_i:=L_i-a_iG~nef, \\ a_i\geq0,~ a_0\leq a\end{array}}} }\{\sum_{j\neq i}(L'_i.L'_j)-\#\{L'_i\} 
+\sum_{i\neq 0} a_i^2+\sum_{i\neq0}a_iG.(L'_i+2\sum_{j\neq i}L'_j)\\ \qquad\qquad\qquad\qquad+2a_0G.L+(a^2_0-(\displaystyle{\sum_{i\neq0}}a_i)^2)e+1\}
\\ \leq\displaystyle{\min_{\tiny{\begin{array}{c}\sum L_i=L-a_0G\\ L'_i:=L_i-a_iG~nef, \\ a_i\geq0,~ a_0\leq a\end{array}}} }\{\sum_{j\neq i}(L'_i.L'_j)-\#\{L'_i\} 
+\sum_{i\neq 0} a_i^2+\sum_{i\neq0}a_iG.(L-a_0G+\sum_{j\neq i}L'_j)\\ \qquad\qquad\qquad\qquad+2a_0G.L+a^2_0e+1\}
\\ =\displaystyle{\min_{\tiny{\begin{array}{c}\sum L_i=L-a_0G\\ L'_i:=L_i-a_iG~nef, \\ a_i\geq0,~ a_0\leq a\end{array}}} }\{\sum_{j\neq i}(L'_i.L'_j)-\#\{L'_i\} 
+\sum_{i\neq 0} a_i^2+\sum_{i\neq0,j\neq i}a_iG.L'_j+(\sum_{i\geq0}a_i)G.L\\ \qquad\qquad\qquad\qquad+a_0G.L+a_0(\sum_{i\geq0}a_i)e+1\}
\end{array}\end{equation}

We want $dim~(\ml\setminus\ml^s)\leq L^2-1$.  

Assume $L'_i=n_iF$ for all $i$, then $\sum_{i\geq 0}a_i=a$.  If moreover $a_i=0$ for $i\neq0$, then $a_0=a$ and $-\#\{L'_i\}+2a_0G.L+a_0^2e+1=1+2a(b-ae)-b+a^2e=b(a-1)+a(b-ae)+1\geq 5$ since $a,b\geq 2$ and $b>ae$.  If $\exists~a_{k_0}\neq 0$ for $k_0\neq 0$, then $-\#\{L'_i\} 
+\sum_{i\neq 0} a_i^2+\sum_{i\neq0,j\neq i}a_iG.L'_j+aG.L+1\geq a(b-ae)+1\geq 3.$

Assume $\exists~ L'_{i_0}\neq nF$, then $L'_{i_0}.L'_j\geq1$ for $L'_j$ nef hence $\sum_{j\neq i}(L'_i.L'_j)\geq 2(\#\{L'_i\}-1)$.  If $a_0\geq1$, then $2a_0G.L+a^2_0e\geq 3$ and hence    
 $\sum_{j\neq i}(L'_i.L'_j)-\#\{L'_i\}+2a_0G.L+a^2_0e+1\geq3$.  If $a_0=0$, then $\#\{L'_i\}\geq 2$ and either $\exists~ L'_{i_0},L'_{j_0}$, such that $L'_{i_0}.L_j\geq 1,~L'_{j_0}.L'_j\geq1$ for $L'_j$ nef; or $\exists ~L'_{i_0},$ such that $L'_{i_0}.L'_j\geq2$ for $L'_j$ nef; or $\exists~a_{k_0}\neq 0$ for $k_0\neq 0$.  Then we have
$$\sum_{j\neq i}(L'_i.L'_j)+\sum_{i\neq 0} a_i^2+\sum_{i\neq0,j\neq i}a_iG.L'_j+(\sum_{i\geq0}a_i)G.L\geq 2(\#\{L'_i\}-1)+2$$ and 
$$\sum_{j\neq i}(L'_i.L'_j)-\#\{L'_i\}+\sum_{i\neq 0} a_i^2+\sum_{i\neq0,j\neq i}a_iG.L'_j+(\sum_{i\geq0}a_i)G.L+1\geq \#\{L_i\}+1\geq 3.$$  

If $e=0$, then easy to see 
\begin{eqnarray}\label{ezero}dim~(\ml\setminus\ml^s)&=&\max_{\tiny{\begin{array}{c}\sum_iL_i=L\\ L_i~nef\end{array}}}\{\sum_{i}L^2_i+\#\{L_i\}\}\nonumber\\
&\leq& L^2-\min_{\tiny{\begin{array}{c}\sum_iL_i=L\\ L_i~nef\end{array}}}\{\sum_{j\neq i}L_iL_j-\#\{L_i\}\} \leq L^2-2.\end{eqnarray}
Therefore the complement of $\ml^s$ inside $\ml$ is of codimension $\geq 3$ and hence $\ml\setminus\mli$ is of codimension $\geq2$. 
 
Because $\Omega_L\setminus\Omega_L^{int}$ is of codimension $\geq2$ and $\ls$ contains smooth curves, sheaves not locally free on their supports form a subset of codimension $\geq 2$ inside $\Omega_L$, hence Proposition 4.2.11 in \cite{Yuan1} applies and then the dualizing sheaf of $\ml$ is isomorphic to $\pi^{*}\mo_{\ls}(L.K_X)$.
Moreover since $\ml$ is normal and integral, and the complement of $\ls^{int}$ inside $\ls$ is of codimension $\geq2$, Theorem 4.3.1 in \cite{Yuan1} and Proposition 4.3 in \cite{Yuan5} apply and we obtain that $\pi_{*}\z_L\cong\mo_{\ls}$ and $R^i\pi_{*}\z_L^r=0$ for all $i,r>0$.  

The lemma is proved.
 \end{proof}
 \begin{rem}\label{com}Let $L$ be as in Corollary \ref{inno}.  Since $\pi_{*}\z_L\cong\mo_{\ls}$ and $R^i\pi_{*}\z_L^r=0$ for all $i,r>0$, $H^i(\z_L(n))=0$ for all $i>0$ and $n\geq 0$.  Hence we already know that the map $g_L$ in (\ref{sddim}) is surjective in this case.  
 \end{rem}
\begin{proof}[Proof of Statement (1) of Theorem \ref{ruled}]By Corollary \ref{inno}, the strange duality map $SD_{c_2^1,u_L}$ in (\ref{rank1}) is a map between two vector spaces of same dimension, while $L$ is in case (1) of the theorem.  The argument proving Corollary 4.3.2 in \cite{Yuan1} applies and hence $SD_{c_2^1,u_L}$ is an isomorphism.  Statement (1) is proved.
\end{proof}

To prove Statement (2) and (3), we need to introduce more stacks.
\begin{defn}For two integers $k>0$ and $i$, we define $\mm_{k,i}^{int}(L,\chi)$ to be the (locally closed) substack of $\mm^{int}(L,\chi)$ parametrizing sheaves $\mf\in\mm^{int}(L,\chi)$ such that $h^1(\mf(-iK_X))=k$ and $h^1(\mf(-nK_X))=0,~\forall~ n>i.$  Let $M^{int}_{k,i}(L,\chi)$ be the image of $\mm_{k,i}^{int}(L,\chi)$ in $M^{int}(L,\chi)$.

Define $\mw_{k,i}^{int}(L,\chi)$ to be the (locally closed) substack of $\mm^{int}(L,\chi)$ parametrizing sheaves $\mf\in\mm^{int}(L,\chi)$ with  $h^0(\mf(-iK_X))=k$ and $h^0(\mf(-nK_X))=0,~\forall ~n<i$.  Let $W^{int}_{k,i}(L,\chi)$ be the image of $\mw_{k,i}^{int}(L,\chi)$ in $M^{int}(L,\chi)$.
\end{defn}
\begin{rem}\label{bound}Since $L.K_X<0$, for fixed $\chi$, $\mm_{k,i}^{int}(L,\chi)$ is empty except for finitely many pairs $(k,i)$.  We don't define $\mm^d_{k,i}(L,\chi)\subset \mm^d(L,\chi)$ because $L$ may not be $K_X$-negative (see Definition 2.1 in \cite{Yuan4}) and the analogous definition may not behave well. 
\end{rem}
\begin{rem}\label{duke}By sending each sheaf $\mf$ to its dual $\E xt^1(\mf,K_X)$, we get an isomorphism $\mm_{k,i}^{int}(L,\chi)\xrightarrow{\cong}\mw^{int}_{k,-i}(L,-\chi)$.
\end{rem}
By Proposition 5.5 and Remark 5.6 in \cite{Yuan4}, we have
\begin{prop}\label{dnki}$dim~\mm_{k,i}^{int}(L,\chi)\leq L^2+iK_X.L-\chi-k$ and $dim~\mw_{k,i}^{int}(L,0)\leq L^2-iK_X.L+\chi-k$.  Hence $dim~M_{k,i}^{int}(L,0)\leq L^2+1+iK_X.L-\chi-k$ and $dim~W_{k,i}^{int}(L,0)\leq L^2+1-iK_X.L+\chi-k$.
\end{prop}
\begin{coro}\label{coddz}Let $X=\Sigma_e$ and $L=aG+bF$ ample with $\min\{a,b\}\geq2$.  
Let $D_{\z_L}^{int}:=D_{\z_L}\cap\mli$.  Then $dim~D_{\z_L}\setminus D_{\z_L}^{int}\leq L^2-2$, and $dim~\cd_{\z_L}\setminus \cd_{\z_L}^{int}\leq L^2-3$ with $\cd_{\z_L}$ ($\cd_{\z_L}^{int}$, resp.) the preimage of $D_{\z_L}$ ($D_{\z_L}^{int}$, resp.) inside $\mm(L,0)$.
\end{coro}
\begin{proof}
We have shown that $\ml\setminus\ml^s$ is of dimension $\leq L^2-2$.  Then we only need to show $dim~(\cd_{\z_L}\setminus \cd_{\z_L}^{int})\leq L^2-3$.  Let $\mc_{L_1,~L_2}$ with $L_1+L_2=L$ be the stack parametring sheaves $\mf\in \cd_{\z_L}$ with supports $C_{\mf}=C_{L_1}+C_{L_2}$ such that $C_{L_i}\in |L_i|^{int}$ for $i=1,2$.  By (\ref{nrdim}) and (\ref{redd}), we only need to show the stacks $\mc_{2G+(b-1)F,~F}$ and $\mc_{(a-1)G+(ae+1)F,~G}$ is of dimension $\leq L^2-3$.

Let $\mf\in\mc_{2G+(b-1)F,~F}$.  Then we have the following exact sequence
\begin{equation}\label{ae2}0\ra\mf_1\ra\mf\ra\mf_2\ra0,\end{equation}
where $\mf_2$ is the torsion free part of $\mf\otimes\mo_{C_F}$ and $\mf_1\in\mm^{int}(2G+(b-1)F,\chi_1)$ with $\chi_1\leq 0$.  Notice that $\mf_1\otimes\mo_X(F)$ is a quotient of $\mf$, hence $\chi_1+2\geq 0$.  Also $\mf_2\otimes\mo_X(-2G-(b-1)F)$ is a subsheaf of $\mf$ and hence $\mf_2\cong\mo_{\p^1}$ or $\mo_{\p^1}(-1)$.  Let $\mc_{2G+(b-1)F,~F}^0\subset \mc_{2G+(b-1)F,~F}$ consist of $\mf$ in (\ref{ae2}) with $H^0(\mf_1)=0$.  $\mf_1\in\bigcup_{i\leq0}\mw^{int}_{k,i}(2G+(b-1)F,\chi_1)$ if $\mf\in\mc_{2G+(b-1)F,~F}\setminus \mc^0_{2G+(b-1)F,~F}$.  Therefore
\begin{equation}\begin{array}{l}dim ~\mc_{2G+(b-1)F,~F}\setminus \mc^0_{2G+(b-1)F,~F}\\
\leq (2G+(b-1)F).F+dim~\displaystyle{\bigcup_{i\leq0}}~\mw^{int}_{k,i}(2G+(b-1)F,\chi_1)\\
\leq (2G+(b-1)F)^2-1+\chi_1+2\leq4b-4e-3=L^2-3.\end{array}\end{equation}

Denote by $g_L$ the arithmetic genus of curves in $\ls$.  If $\mf\in \mc^0_{2G+(b-1)F,F}$, then there is a injection $\mo_{C_{\mf}}\hookrightarrow\mf$ with cokernel $\mo_{Z_{\mf}}$, where $Z_{\mf}$ is a 0-dimensional subscheme of $C_{\mf}$ with length $g_L-1$.  We have $\text{ext}^1(\mo_{Z},\mo_{C})=dim~Aut(\mo_{Z})=h^0(\mo_Z)=g_L-1$ for all $Z\subset C$.  Hence for a fixed curve $C$ and $[Z]\in C^{[g_L-1]}$, there are finitely many possible choices for $\mf$ lying in the following sequence
\[0\ra\mo_C\ra\mf\ra\mo_Z\ra0.\]
Hence the fiber of the projection $\mc^0_{2G+(b-1)F,F}\ra |2G+(b-1)F|\times|F|$ over a curve $C$ is of dimension no larger than 
$$dim~C^{[g_L-1]}+\text{ext}^1(\mo_Z,\mo_C)-dim~Aut(\mo_{C})\times Aut(\mo_Z)=dim~C^{[g_L-1]}-1.$$  
Therefore
\[\begin{array}{l}dim~\mc^0_{2G+(b-1)F,F}\\
\leq dim~|2G+(b-1)F|\times|F|-1+\displaystyle{\max_{\mf\in\mc^0_{2G+(b-1)F,F}}}dim~C_{\mf}^{[g_L-1]}\\
=3b-3e-1+\displaystyle{\max_{\mf\in\mc^0_{2G+(b-1)F,F}}}dim~C_{\mf}^{[g_L-1]}\\
=4b-4e-3+(\displaystyle{\max_{\mf\in\mc^0_{2G+(b-1)F,F}}}dim~C_{\mf}^{[g_L-1]}-(g_L-1))\\
=L^2-3+(\displaystyle{\max_{\mf\in\mc^0_{2G+(b-1)F,F}}}dim~C_{\mf}^{[g_L-1]}-(g_L-1)).\end{array}\]
The only thing left to prove is $dim~C_{\mf}^{[g_L-1]}\leq g_L-1$ for all $C_{\mf}$, and this follows from that $C_{\mf}$ only have isolated planner singularities and the result of Iarrobino (Corollary 2 in \cite{Ia}).  

Analogously we can show that $dim~\mc_{(a-1)G+(ae+1)F,~G}\leq L^2-3$.  
The corollary is proved. 
 \end{proof}

\begin{proof}[Proof of Statement (2) and (3) of Theorem \ref{ruled}]  The proof has 7 steps and we check all conditions in $\cb$ and $\cb'$ one by one as follows. 

\emph{Step 1: $\cb$-(1).}  

Since $\ml$ is integral and $D_{\z_L}$ is a divisor on it, to show $\cb$-(1) it is enough to show $dim~(D_{\z_L}\setminus D_{\z_L}^o)\leq L^2-1.$  By Corollary \ref{coddz}, it is enough to show $dim~(D^{int}_{\z_L}\setminus D_{\z_L}^o)\leq L^2-1$.  By definition 
\[(D^{int}_{\z_L}\setminus D_{\z_L}^o)\subset \bigcup_{\begin{array}{c}k\geq2,i=0\\~or~i\geq1\end{array}} M_{k,i}^{int}(L,0).\]
Therefore we have $\cb$-(1) is fulfilled by Proposition \ref{dnki}.

\emph{Step 2: $\cb$-(2).}

Assume $L=aG+bF$ ample with $\min\{a,b\}\geq 4$.  
Then Lemma \ref{codim} applies to $L+K_X=(a-2)G+(b-e-2)F$ and $\mlk\setminus M^{int}(L\otimes K_X,0)$ is of codimension $\geq2$.  $\mlk$ satisfies the ``condition $S_2$ of Serre" because it is normal  by Corollary \ref{inno}.  Hence to prove $\cb$-(2) is fulfilled, it is enough to show $M^{int}(L\otimes K_X,0)\setminus\mlk^o$ is of dimension $\leq (L+K_X)^2-1$.
Since we have
\[M^{int}(L\otimes K_X,0)\setminus\mlk^{o}=\bigcup_{i\leq -1} W^{int}_{k,i}(L\otimes K_X,0),\] 
by Proposition \ref{dnki} we have \[dim~M^{int}(L\otimes K_X,0)\setminus\mlk^{o}\leq (L+K_X)^2+K_X.(L+K_X)\leq (L+K_X)^2-1.\]
Hence $\cb$-(2). 

\emph{Step 3: $\cb$-(3).}

To check that $\cb$-(3) holds, it is enough to show the following three statements.
\begin{enumerate} 
\item $dim~Q_1^o=2\ell-L.K_X=L^2$;
\item $\rho^{-1}_1(\rho_1(Q_1^o))\setminus Q_1^o$ is of dimension $\leq 2\ell-L.K_X-2=L^2-2$;
\item $H_{\ell}\setminus\rho_1(Q_1^o)$ is of dimension $\leq 2\ell-2=L^2+L.K_X-2$.
\end{enumerate}
Let $s>0,t\geq0$, and define
$$Q_1^{s,t}:=\left\{[f_1:\mi_Z(L\otimes K_X)\twoheadrightarrow \mf_L]\in Q_1 \left|h^1(\mf_L)=s,~h^1(\mf_L(-K_X))=t\right.\right\},$$
$$H_{\ell}^{L,s,t}:=\left\{\mi_Z\in H_{\ell}\left|h^1(\mi_Z(L\otimes K_X))=s-1,~h^1(\mi_Z(L))=t.\right.\right\}.$$
Then $Q_1^{s,t}=\rho_1^{-1}(H_{\ell}^{L,s,t})$, $\rho_1(Q_1^o)\subset H_{\ell}^{L,1,0}$ and $\rho^{-1}_1(\rho_1(Q_1^o))\subset Q_1^{1,0}$.  For $d$ large enough, we have the classifying map $Q_1^{s,t}\xrightarrow{\phi^{L,s,t}_{L}}\mm^d(L,0).$
In particular when $s=1$, $\phi^{L,1,t}_{L}(Q_1^{1,t})\subset\mm(L,0)$, hence $\phi^{L,1,t}_{L}(Q_1^{1,t})\subset\cd_{\z_L}$.  This is because for every $\mf\in\mm^d(L,0)$, if $h^0(\mf_L)=1$ and there is a torsion free extension of $\mf_L$ by $K_X$, then $\forall~\mf'\subsetneq \mf$, $h^0(\mf')\leq1$ and $h^1(\mf')\geq1$ hence then $\chi(\mf')\leq0$ and $\mf$ is semistable.  
The fiber of $\phi_{L}^{L,s,t}$ at $\mf_{L}$ is contained in $\Ext^1(\mf_{L},K_X)$, and hence
\[dim~Q_1^{o}=1+dim(\cd_{\z_L}\cap(\mm^{int}(L,0)\setminus\bigcup_{k\geq2,i=0~or~i>0}\mm^{int}_{k,i}(L,0)))=L^2.\]
\begin{equation}\label{count}dim~(\bigcup_{t\geq0}Q_1^{1,t})\setminus Q_1^o\leq 1+dim~((\cd_{\z_L}\setminus \cd^{int}_{\z_L})\cup\bigcup_{i>0}\mm^{int}_{k,i}(L,0))\leq L^2-2,\end{equation}
where the last inequality is because of Corollary \ref{coddz} and Proposition \ref{dnki}.

By (\ref{dd}) $Q_1^{s,t}\cong\p(p_{*}(\I_{\ell}\otimes q^*L)|_{H_{\ell}^{L,s,t}})$, where $p_{*}(\I_{\ell}\otimes q^*L)$ is a vector bundle of rank $h^0(\mi_Z(L))=t+1-L.K_X$ over $H_{\ell}^{L,s,t}$.  Hence 
\begin{equation}\label{q1h}dim~Q_1^{s,t}=dim~H_{\ell}^{L,s,t}+t-L.K_X.\end{equation}
Hence (\ref{count}) impies $dim~(\bigcup_{t\geq0}H_{\ell}^{L,1,t})\setminus\rho_1(Q_1^o)\leq L^2+L.K_X-2$.  Hence we only need to show $dim~H_{\ell}\setminus(\bigcup_{t\geq0} H_{\ell}^{L,1,t})\leq 2\ell-2$, i.e. $dim~H_{\ell}^{L,s,t}\leq 2\ell-2$ for all $s\geq2$.

$p_{*}(\I_{\ell}\otimes q^*(L\otimes K_X))$ is a vector bundle of rank $h^0(\mi_Z(L\otimes K_X))=s$ over $H_{\ell}^{L,s,t}$.  By (\ref{ddt}) $\p(p_{*}(\I_{\ell}\otimes q^*(L\otimes K_X))|_{H_{\ell}^{L,s,t}})$ is a locally closed subscheme inside $Q_2$.  For $d$ big enough, there is a classifying map 
\[\p(p_{*}(\I_{\ell}\otimes q^*(L\otimes K_X))|_{H_{\ell}^{L,s,t}})\xrightarrow{\phi^{L,s,t}_{L\otimes K_X}}\mm^d(L\otimes K_X,0).\]  
If $s\geq 2$, then the image of $\phi_{L\otimes K_X}^{L,s,t}$ is contained in 
$$(\mm^d(L\otimes K_X,0)\setminus\mm^{int}(L\otimes K_X,0))\cup\bigcup_{\begin{array}{c}i=0,k=s-1\\or~i<0\end{array}}\mw^{int}_{k,i}(L\otimes K_X,0).$$   
The fiber of $\phi_{L\otimes K_X}^{L,s,t}$ at $\mf_{L\otimes K_X}$ is contained in $\Ext^1(\mf_{L\otimes K_X},\mo_X)$.  
If $\mf_{L\otimes K_X}\in\mw_{s-1,0}^{int}(L\otimes K_X,0)$, then $h^0(\mf_{L\otimes K_X}(K_X))=0$ and $\text{ext}^1(\mf_{L\otimes K_X},\mo_X)=-(L+K_X).K_X$.
If $\mf_{L\otimes K_X}\not\in\mw_{s-1,0}^{int}(L\otimes K_X,0)$, then since $-K_X-G$ is very ample, by (\ref{ddt}) we have
\begin{eqnarray}\label{dgf1}h^0(\mf_{L\otimes K_X}(K_X))&=&h^0(\mi_Z(L\otimes K_X^{\otimes 2}))\nonumber\\ & \leq& h^0(\mi_Z(L\otimes K_X\otimes\mo_X(-G)))-1
\nonumber\\ &\leq& h^0(\mi_Z(L\otimes K_X))-1=s-1.\end{eqnarray}
Hence $\text{ext}^1(\mf_{L\otimes K_X},\mo_X)\leq s-1-(L+K_X).K_X$.  Hence for $s\geq 2$
\begin{equation}\label{dh1}\begin{array}{l}dim~H_{\ell}^{L,s,t}+s-1 = ~dim~\p(p_{*}(\I_{\ell}\otimes q^*(L\otimes K_X))|_{H_{\ell}^{L,s,t}})\\
\leq \max\{dim~\mw^{int}_{s-1,0}(L\otimes K_X,0)-(L+K_X).K_X,~~ s-1-(L+K_X).K_X+\\ ~~~~ dim~(\mm^d(L\otimes K_X,0)\setminus\mm^{int}(L\otimes K_X,0)\cup\displaystyle{\bigcup_{i<0}}\mw^{int}_{k,i}(L\otimes K_X,0))\}\\
\leq\max\{(L+K_X).L-1, ~(L+K_X).L-3+s\}=(L+K_X).L-3+s.\end{array}
\end{equation}
Hence $dim~H_{\ell}^{L,s,t}\leq 2\ell-2$ for all $s\geq2$.  
Hence $\cb$-(3) is fulfilled.

\emph{Step 4: $\cb$-(4).}

$\cb$-(4) can be shown analogously: $Q_2^o$ is obviously nonempty and there is a classifying map $Q_2\xrightarrow{\phi^{L\otimes K_X}_{L\otimes K_X}}\mm^d(L\otimes K_X,0)$ with fiber over $\mf_{L\otimes K_X}$ contained in $\Ext^1(\mf_{L\otimes K_X},\mo_X)$.   $dim~\rho^{-1}_2(\rho_2(Q_2^o))\setminus Q_2^o\leq dim~Q_2^o-2$ because 
\[\begin{array}{l}~~~~dim~\phi^{L\otimes K_X}_{L\otimes K_X} (\rho^{-1}_2(\rho_2(Q_2^o))\setminus Q_2^o)\\ \leq 
dim~\mm^d(L\otimes K_X,0)\setminus\mm^{int}(L\otimes K_X,0)\cup\bigcup_{i\leq -1}\mw^{int}_{k,i}(L\otimes K_X,0) \\ \leq (L+K_X)^2-2,\end{array}\]
and $\text{ext}^1(\mf_{L\otimes K_X},\mo_X)=-K_X.(L+K_X)$ for all $\mf_{L\otimes K_X}\in \phi^{L\otimes K_X}_{L\otimes K_X} (\rho^{-1}_2(\rho_2(Q_2^o))$.
 
Statement (3) is proved.

\emph{Step 5: $\cb'$-$(3')$ and $\cb'$-$(2'a)$.}

Now we prove Statement (2) of the theorem.  We need to check conditions in $\cb'$ hold.  With no loss of generality, we ask $L=aG+bF$ with $b\geq a$.  Then in this case $L+K_X=mF$ or $G+nF$ with $n>0$ for $e=0$, and $n\geq 2e-1$ for $e\geq 1$.  
Then $\mf_{L\otimes K_X}$ is semistable $\Leftrightarrow~H^0(\mf_{L\otimes K_X})=0$.  Hence $H_{\ell}'\subset \bigcup_{t\geq0} H^{L,1,t}_{\ell}$ and by (\ref{count}) we have $\cb'$-$(3')$.  Notice that (\ref{count}) holds for $L=aG+bF$ ample with $\min\{a,b\}\geq 2$.

Also $M(L\otimes K_X,0)\cong |L\otimes K_X|$ and $\mlk'\cong|L\otimes K_X|'$.  Then easy to check $\cb'$-$(2'a)$ holds.

\emph{Step 6: $\cb'$-$(2'b)$.}

Now we check $\cb'$-$(2'b)$.  First let $L\otimes K_X=G+nF$ with $|G+nF|^{int}\neq\emptyset$. 
Recall the commutative diagram in (\ref{pie})
\begin{equation}\label{rpie}\xymatrix{\p(\mv') &P_2'\ar[l]_{~~~~\supseteq}\ar[r]^{\sigma'_2}\ar[d]_{f_{Q_2}'} &\Omega'_{L\otimes K_X}\ar[d]^{f'_M}\\ & Q_2'\ar[r]_{g_2'~~~~~~~} &\mlk' }.\end{equation}   
where $\mv'=\mathscr{E}xt_p^1(\Q_{L\otimes K_X}|_{\Omega'_{L\otimes K_X}},q^{*}\mo_X)$ with $\Q_{L\otimes K_X}$ the universal quotient over $\Omega_{L\otimes K_X}$.  $\mv'$ is locally free of rank $-(L+K_X).K_X$ on $\Omega'_{L\otimes K_X}$.  $P'_2\subset\p(\mv')$ parametrizing torsion free extensions of $\Q_{\mathfrak{s}}$ by $\mo_X$ for all $\mathfrak{s}\in \Omega'_{L\otimes K_X}$ and $f'_{Q_2}: P'_2\ra Q'_2$ is the classifying map and also a principal $PGL(V_{L\otimes K_X})$-bundle for some vector space $V_{L\otimes K_X}$.

To show the complement of $P_2'$ inside $\p(\mv')$ is of codimension $\geq2$, it is enough to show for every $\mf_{L\otimes K_X}\in\mm^R(L\otimes K_X,0)$, $H^0(\mf_{L\otimes K_X}(K_X))=0$ with support $C_{\mf_{L\otimes K_X}}=C^1_{\mf_{L\otimes K_X}}\cup C^2_{\mf_{L\otimes K_X}}$ such that $C^1_{\mf_{L\otimes K_X}}\in|F|$ and $C^2_{\mf_{L\otimes K_X}}\in|G+(n-1)F|^{int}$ , there is a torsion free extension in $\Ext^1(\mf_{L\otimes K_X},\mo_X)$.  $\forall~\mf'_{L\otimes K_X}\subsetneq\mf_{L\otimes K_X}$, $\Ext^1(\mf_{L\otimes K_X}/\mf'_{L\otimes K_X},\mo_X)$ can be view as a subspace of $\Ext^1(\mf_{L\otimes K_X},\mo_X)$.  There is a torsion free extension in $\Ext^1(\mf_{L\otimes K_X},\mo_X)$ $\Leftrightarrow$ $\text{ext}^1(\mf_{L\otimes K_X}/\mf'_{L\otimes K_X},\mo_X)< \text{ext}^1(\mf_{L\otimes K_X},\mo_X)$, $\forall~\mf'_{L\otimes K_X}\subsetneq\mf_{L\otimes K_X}$. 
Now we have that $C_{\mf_{L\otimes K_X}}=C^1_{\mf_{L\otimes K_X}}\cup C^2_{\mf_{L\otimes K_X}}$, $C^i_{\mf_{L\otimes K_X}}\cong\p^1$ and $deg(K_X|_{C^i_{\mf_{L\otimes K_X}}})<0$, for $i=1,2$.  Therefore
$\forall~\mf'_{L\otimes K_X}\subsetneq\mf_{L\otimes K_X}$, either $\Ext^1(\mf_{L\otimes K_X}/\mf'_{L\otimes K_X},\mo_X)=0$ or $\Ext^2(\mf_{L\otimes K_X}/\mf'_{L\otimes K_X},\mo_X)=0$.  Hence the map $\Ext^1(\mf_{L\otimes K_X}/\mf'_{L\otimes K_X},\mo_X)\hookrightarrow \Ext^1(\mf_{L\otimes K_X},\mo_X)$ can not be surjetive.  The reason is that $\text{ext}^1(\mf_{L\otimes K_X},\mo_X)=2n+2-e>0$ and $\Ext^1(\mf'_{L\otimes K_X},\mo_X)\neq 0$ since $\chi(\mf'_{L\otimes K_X}(K_X))<0$.

If $L\otimes K_X=F$, then $\cb'$-$(2'b)$ is obvious.  Let $|L\otimes K_X|^{int}=\emptyset$, i.e. $L\otimes K_X=nF$ with $n>1$.  In this case $|L\otimes K_X|'=|L\otimes K_X|$.  $\Omega_{L\otimes K_X}'=\Omega_{L\otimes K_X}$.  In order to show $dim~\p(\mv')\setminus P_2'\leq dim~P_2'-2$, it is enough to show for every $\mf_{L\otimes K_X}$ semistable, $\p(\Ext^1(\mf_{L\otimes K_X},\mo_X)\setminus\Ext^1(\mf_{L\otimes K_X},\mo_X)^{tf})$ is of dimension $\leq -K_X.(L+K_X)-3$, where 
$\Ext^1(\mf_{L\otimes K_X},\mo_X)^{tf})$ is the subset parametrzing torsion free extensions.  We have 
\[\Ext^1(\mf_{L\otimes K_X},\mo_X)\setminus\Ext^1(\mf_{L\otimes K_X},\mo_X)^{tf}=\bigcup_{\mf'_{L\otimes K_X}\subsetneq\mf_{L\otimes K_X}}\Ext^1(\mf_{L\otimes K_X}/\mf'_{L\otimes K_X},\mo_X)\]
Since $L\otimes K_X=nF$, $\forall~\mf'_{L\otimes K_X}\subsetneq\mf_{L\otimes K_X}$, either $\Ext^1(\mf_{L\otimes K_X}/\mf'_{L\otimes K_X},\mo_X)=0$ or $\Ext^2(\mf_{L\otimes K_X}/\mf'_{L\otimes K_X},\mo_X)=0$.  Hence we only need to show that $\forall~\mf'_{L\otimes K_X}\subsetneq\mf_{L\otimes K_X}$ such that $\Ext^1(\mf_{L\otimes K_X}/\mf'_{L\otimes K_X},\mo_X)\neq 0$, $\text{ext}^1(\mf'_{L\otimes K_X},\mo_X)\geq 2$.  It is enough to show $\text{ext}^1(\mf'_{L\otimes K_X},\mo_X)\geq 2$ for every $\mf'_{L\otimes K_X}\subsetneq\mf_{L\otimes K_X}$ with $C_{\mf'_{L\otimes K_X}}$ integral.  On the other hand $C_{\mf'_{L\otimes K_X}}\cong\p^1$ if integral, and also $deg(K_X|_{C_{\mf'_{L\otimes K_X}}})=-2$.  Hence $\text{ext}^1(\mf'_{L\otimes K_X},\mo_X)=h^1(\mf'_{L\otimes K_X}(K_X))\geq2$ because $\chi(\mf'_{L\otimes K_X}(K_X))\leq -2$.

\emph{Step 7: $\cb'$-$(4')$.}

$\cb'$-$(4')$ is the last thing left to check.  
$$Q_2':=\left\{[f_2:\mi_Z(L\otimes K_X)\twoheadrightarrow \mf_{L\otimes K_X}]\in Q_2\left|\begin{array}{l}
\mf_{L\otimes K_X}~is~semistable,\\ h^0(\mf_{L\otimes K_X}(K_X))=0,~and\\ 
Supp(\mf_{L\otimes K_X})\in |L\otimes K_X|'.
\end{array}\right.\right\}.$$
In this case $\mf_{L\otimes K_X}$ is semistable $\Leftrightarrow$ $H^0(\mf_{L\otimes K_X})=0$.   $h^0(\mi_Z(L\otimes K_X))=1$ for all $\mi_Z\in\rho_2(Q_2')$, hence $\rho_2|_{Q_2'}$ is bijective and hence an isomorphism, therefore $Q_2'\cong \rho_2^{-1}(\rho_2(Q'_2))$ and $\cb'$-$(4')$ holds.  

The proof of Theorem \ref{ruled} is finished.
\end{proof}

\end{document}